\DeclareSymbolFont{AMSb}{U}{msb}{m}{n}
\documentclass[noamsfonts,11pt]{amsart}
\usepackage[bitstream-charter]{mathdesign}
\usepackage[margin=1.25in]{geometry}

\usepackage{comment}
\usepackage{latexsym, amsthm, amsmath,enumerate}
\usepackage{graphicx,pict2e}
\usepackage{multirow,multicol,mathtools}
\usepackage[svgnames]{xcolor}
\usepackage{tikz}
\usepackage{ytableau}
\usepackage[all,cmtip]{xy}
\usepackage{young}

\newlength\circlesize
\newcommand*\circled[1]{\tikz[baseline=(char.base)]{
		\node[shape=circle,draw,inner sep=2pt] (char) {#1};}}

\allowdisplaybreaks[1]

\newcommand{\NN}{\mathbb{N}}
\newcommand{\QQ}{\mathbb{Q}}

\usepackage[cmtip,all]{xy}
\newcommand{\longsquiggly}{\xymatrix{{}\ar@{~>}[r]&{}}}

\newcommand{\SSYT}{\mathrm{SSYT}}
\newcommand{\ShST}{\mathrm{ShST}}

\newcommand{\std}{\mathrm{std}}

\newcommand{\rectify}{\mathrm{rect}}

\newcommand{\height}{\mathrm{ht}}

\newcommand{\B}{\mathcal{B}}

\newcommand{\wt}{\mathrm{wt}}

	\newtheorem{lemma}{Lemma}
	\newtheorem{theorem}[lemma]{Theorem}
	
	\newtheorem{proposition}[lemma]{Proposition}
	\newtheorem{corollary}[lemma]{Corollary}
	
	\newtheorem{question}[lemma]{Question}

	\theoremstyle{definition}
	\newtheorem{example}[lemma]{Example}
	\newtheorem{definition}[lemma]{Definition}
	\newtheorem{remark}[lemma]{Remark}

	\numberwithin{equation}{section}
	\numberwithin{figure}{section}
	\numberwithin{table}{section}
	\numberwithin{lemma}{section}
	
	\newcommand{\defn}[1]{{\bf #1}}
	
	\newcommand{\east}[1]{\ensuremath{\xrightarrow{\ #1\ }}}
	\newcommand{\west}[1]{\ensuremath{\xleftarrow{\ #1\ }}}
	\newcommand{\north}[1]{\ensuremath{\big \uparrow \!\text{\raisebox{.1ex}{\scriptsize $#1$}}}}
	\newcommand{\south}[1]{\ensuremath{\big \downarrow \!\text{\raisebox{.1ex}{\scriptsize $#1$}}}}
	
	\newcommand{\stepnorth}[2]{\vector(0,1){.92}\put(-.23,.4){\scriptsize$#1$}\put(0,1){#2}}
	\newcommand{\stepnorthshift}[2]{\put(-.08,0){\vector(0,1){.95}\put(-.23,.4){\scriptsize$#1$}}\put(0,1){#2}}
	\newcommand{\stepsouth}[2]{\vector(0,-1){.92}\put(.05,-.6){\scriptsize$#1$}\put(0,-1){#2}}
	\newcommand{\stepsouthshift}[2]{\put(.08,0){\vector(0,-1){.92}\put(.05,-.6){\scriptsize$#1$}}\put(0,-1){#2}}
	\newcommand{\stepeast}[2]{\vector(1,0){.92}\put(.4,-.25){\scriptsize$#1$}\put(1,0){#2}}
	\newcommand{\stepeastshift}[2]{\put(0,-.08){\vector(1,0){.92}\put(.4,-.25){\scriptsize$#1$}}\put(1,0){#2}}
	\newcommand{\stepwest}[2]{\vector(-1,0){.92}\put(-.5,.05){\scriptsize$#1$}\put(-1,0){#2}}
	\newcommand{\stepwestshift}[2]{\put(0,.08){\vector(-1,0){.92}\put(-.5,.05){\scriptsize$#1$}}\put(-1,0){#2}}
	\title{A crystal-like structure on shifted tableaux} 
	
	\keywords{Schubert calculus, shifted tableaux, jeu de taquin, crystal base theory}
	\subjclass[2010]{Primary 05E99; Secondary 05E05}
	
	\author{Maria Gillespie}
	\address{
		Mathematics Department \\
		University of California, Davis \\
		Davis, CA}
	\email{mgillespie@math.ucdavis.edu}
	\thanks{The first author was supported by the NSF MSPRF grant PDRF 1604262.}
	
	\author{Jake Levinson}
	\address{
		LaCIM (Laboratoire de combinatoire et d'informatique math\'{e}mathique) \\
		University of Quebec at Montreal \\
		Montreal, QC}
	\email{jakelev@umich.edu}
	\thanks{The second author was supported by a Rackham Predoctoral Fellowship and by NSERC grant PDF-502633.}
	
	\author{Kevin Purbhoo}
	\address{
		Mathematics Department \\
		University of Waterloo
		Waterloo, ON}
	\email{kpurbhoo@uwaterloo.ca}
	
\begin{document}
		
		\begin{abstract}
			We introduce coplactic raising and lowering operators $E'_i$, $F'_i$, $E_i$, and $F_i$ on shifted skew semistandard tableaux. We show that the primed operators and unprimed operators each independently form type A Kashiwara crystals (but not Stembridge crystals) on the same underlying set and with the same weight functions. When taken together, the result is a new kind of `doubled crystal' structure that recovers the combinatorics of type B Schubert calculus: the highest-weight elements of our crystals are precisely the shifted Littlewood-Richardson tableaux, and their generating functions are the (skew) Schur Q functions.
		\end{abstract}
		
		\maketitle
		
		%%%%%%%%%%%%%%%%%%%%%%%%%%%%%%%%%%%%%%%%%%%%%%%%%%%%%%%%%%%%%%%%%%%%%%
		%%%%%%%%%%%%%%%%%%%%%%%%%%%%%%%%%%%%%%%%%%%%%%%%%%%%%%%%%%%%%%%%%%%%%%
		
		\section{Introduction}
		
		A \textbf{crystal base} is a set $\B$ along with certain raising and lowering operators $E_i,F_i:\B\to \B\cup \{\varnothing\}$, functions $\varphi_i,\varepsilon_i:\B\to\mathbb{Z}\cup \{-\infty\}$, and a weight map $\wt:\B\to \Lambda$ where $\Lambda$ is a weight lattice of some Lie type.   The subscripts $i$ range over an index set $I$ corresponding to the simple roots of the root system of $\Lambda$, and the operators $E_i$ and $F_i$ raise and lower the values of $\varphi_i,\varepsilon_i,\wt$ according to the corresponding root vectors.
		
		Crystal bases were first introduced by Kashiwara \cite{Kashiwara} in the context of the representation theory of the quantized universal enveloping algebra $U_q(\mathfrak{g})$ of a Lie algebra $\mathfrak{g}$ at $q=0$.  Since then, their connections to tableau combinatorics, symmetric function theory, and other aspects of representation theory have made crystal operators and crystal bases the subject of much recent study.  (See \cite{Schilling} for an excellent recent overview of crystal base theory.)
		
		\subsection{Ordinary tableaux crystals}\label{sec:ordinary}
		
		The type A crystal base theory can be described entirely in terms of semistandard Young tableaux.  Let $\B = \SSYT(\lambda/\mu,n)$ be the set of all semistandard Young tableaux of a given skew shape $\lambda/\mu$ and with entries from $\{1,\ldots,n\}$ for some fixed $n$.  If $T\in \B$, let $w$ be its row reading word, formed by reading the rows of $T$ from bottom to top.
		
		The functions $E_i$ and $F_i$ on $\B$ are defined directly in terms of the reading word $w$, as follows. First replace each $i$ in $w$ with a right parentheses and each $i+1$ in $w$ with a left parentheses. For instance, if $w=112212112$ and $i=1$, the sequence of brackets is $))(()())($.  After maximally pairing off the parentheses, $E_i(T)$ is formed by changing the first unpaired $i+1$ to $i$, and $F_i(T)$ is formed by changing the last unpaired $i$ to $i+1$ (they are defined to be $\varnothing$ if the operation is impossible.)
		\[\varnothing\ \xleftarrow{\ E\ }\ \ ))\underline{(()())})\ \ \xleftarrow{\ E\ }\ w = ))\underline{(()())}(\ \ \xrightarrow{\ F\ }\ \ )(\underline{(()())}(\ \ \xrightarrow{\ F\ }\ \ ((\underline{(()())}(\ \ \xrightarrow{\ F\ }\ \varnothing
		\]
		The functions $\varphi_i(T)$ and $\varepsilon_i(T)$ can be defined as the smallest $k$ for which $F_i^k(T)=\varnothing$ or $E_i^k(T)=\varnothing$ respectively, and the weight function $\wt(T)$ is simply the weight vector $(m_j)$ where $m_j$ is the number of $j$'s that occur in $T$.
		
		In the case where the tableaux are of straight shape, with entries in $\{1,2\}$, the action simplifies to the following natural chain structure:
		\begin{center}
			\includegraphics{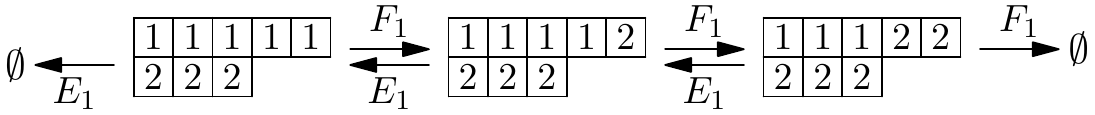}
		\end{center}
		One important property of the $E_i$ and $F_i$ operations is that they are \textbf{coplactic}, that is, they commute with all sequences of jeu de taquin slides. Thus, if we perform the same outwards jeu de taquin slide on the three tableaux above (in the second row, for example), the crystal operators must act in the same way, as shown:
		\begin{center}
			\includegraphics{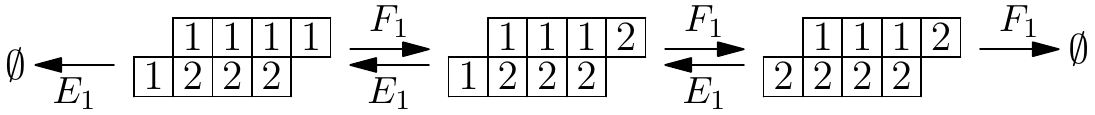}
		\end{center}
		In this sense the operations $E_i$ and $F_i$ are in fact the \textit{unique} coplactic operators that give the natural connected chain structure on rectified shapes containing only $i$, $i+1$.  Notably, Littlewood-Richardson skew tableaux (those that rectify to the highest weight tableau of a given shape) are precisely those for which $E_i(T) = \varnothing$ for all $i$.  Finally, it can be shown that the connected components $C$ of \emph{any} tableau crystal have weights $$\sum_{T\in C}x^{\wt(T)}=s_\nu,$$ where $\nu$ is the common rectification shape of every $T \in C$, and $s_\nu$ is the corresponding Schur function.  Thus, decomposing the tableau crystal $\B = \SSYT(\lambda/\mu,n)$ into its connected components recovers the Schur expansion of the skew Schur function $$s_{\lambda/\mu}=\sum_\nu c^{\lambda}_{\mu\nu}s_\nu.$$
		(See \cite{Schilling} for a more thorough introduction to the above notions.) 
		
		\subsection{Shifted tableaux; results of this paper}
		Despite the elegance of the crystal operators on ordinary semistandard tableaux, a similar structure on \textbf{shifted tableaux} has proven elusive. In \cite{GJKKK}, Grantcharov, Jung, Kang, Kashiwara, and Kim use Serrano's \emph{semistandard decomposition tableaux} \cite{Serrano} to understand the quantum queer superalgebras because, ``unfortunately, the set of shifted semistandard Young tableaux of fixed shape does not have a natural crystal structure.'' In this paper, we provide a potential resolution to this issue by defining a crystal-like structure on shifted tableaux. 
		
		%Our crystals are a new kind of `doubled' type A crystal (Theorem \ref{thm:main-doubled-typeA}); however, their highest weight elements are precisely the type B Littlewood-Richardson shifted tableaux, that is, those whose shifted rectification is the unique highest-weight shifted tableau. Correspondingly, their generating functions are the (skew) Schur Q functions.  
		%%%%%%%%%%%%%%%%
		% MARIA INSERTED THE FOLLOWING HERE:
		
		%  The basic idea is similar to that of ordinary tableau crystals.  
		
		For `straight' (non-skew) shifted tableaux on the alphabet $\{1',1,2',2\}$, there is a natural organization of the tableaux of a given shape (Figure \ref{fig:two-row}), similar to the $F_1$ chains shown above for ordinary tableaux. Namely, if the tableau has two rows, the first row may or may not contain a $2'$, giving two `chains' linked by the operators $F_1$ and $F_1'$. If instead $\lambda$ has only one row, the tableau cannot contain a $2'$, so there is only one chain's worth of tableaux. In this case $F_1=F_1'$.
		
		\begin{figure}
			\begin{center}
				\includegraphics[width=.9\linewidth]{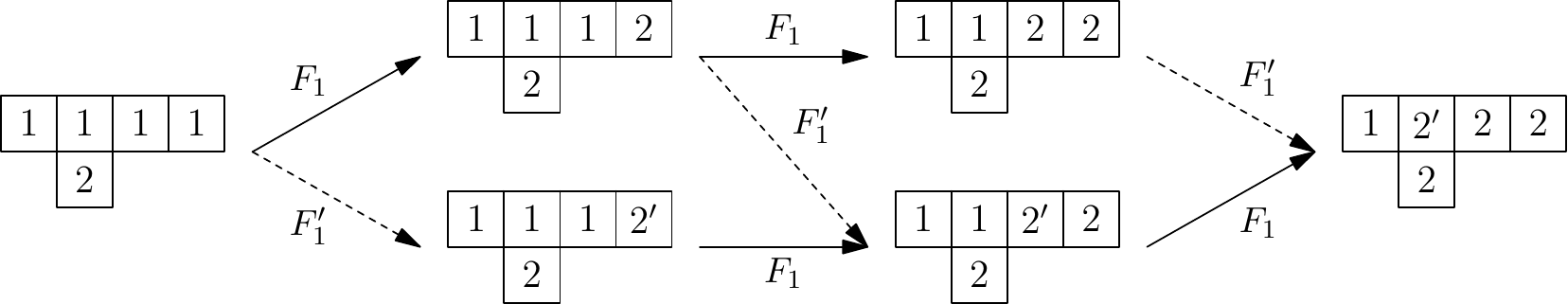} \vspace{1cm}
				
				\includegraphics[width=.85\linewidth]{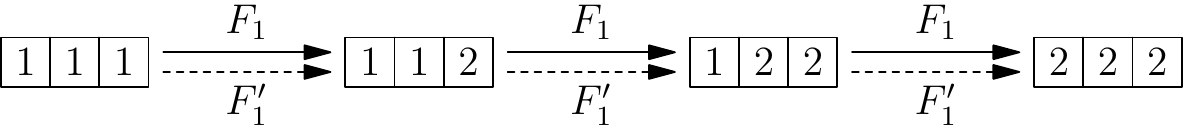}
			\end{center}
			\caption{\label{fig:two-row}The crystals of the form $\ShST(\lambda,2)$ are `two-row' and `one-row' diagrams, organizing the tableaux into `doubled' crystals. {\bf Above}: the tableaux of shapes $(4,1)$ and $(3)$. Wherever an arrow is missing, the corresponding operator is not defined. Reversing the arrows gives the partial inverses $E_1$ and $E_1'$.} 
		\end{figure}
		
		As is the case with ordinary jeu de taquin, the uniqueness of shifted rectification (and more precisely, Haiman's theory of shifted dual equivalence, see \cite{Haiman}) implies that these definitions uniquely extend to coplactic operators on all shifted skew tableaux, giving such tableaux a crystal-like structure. However, a direct description that does not rely on performing jeu de taquin -- analogous to the pairing-parentheses description of $E$ and $F$ on ordinary tableaux -- is far from obvious. The main purpose of this paper is, then, to exhibit a simple combinatorial description of the coplactic operators $E_i,E_i',F_i,F_i'$.
		
		%%%%%%%%%%%%%%%%%%
		% END MOVED SECTION
		%%%%%%%%%%%%%%%%%%
		
		Our main results are as follows. Let $\lambda/\mu$ be a shifted skew shape and let $\ShST(\lambda/\mu,n)$ be the set of shifted semistandard tableaux on the alphabet $\{1' < 1 < 2' < 2 < \cdots < n' < n\}.$
		
		\begin{theorem}\label{thm:main}
			There are combinatorially-defined lowering operators $F_i,F'_i$ ($i = 1, \ldots, n-1$) on $\ShST(\lambda/\mu,n)$, with partial inverse (raising) operators $E_i, E'_i$, depending only on the reading word of the tableau, with the following properties:
			\begin{itemize}
				\item[(i)] They are coplactic for shifted jeu de taquin.
				%\item[(ii)] The lowering operators $F_i,F'_i$ reduce the weight of a tableau by the simple root $\alpha_i$, that is, they change an $i$ or $i'$ to an $i+1$ or $i+1'$ (and possibly prime or unprime other entries).  The raising operators $E_i,E_i'$ do the opposite.
				\item[(ii)] The highest-weight elements (those for which $E_i(T) = E_i'(T) = \varnothing$ for all $i$) are precisely the type B Littlewood-Richardson tableaux.
				\item[(iii)] Let $\B$ be the induced graph on $\ShST(\lambda/\mu,n)$. Then each connected component of $\B$ has a \emph{unique} highest-weight element.
			\end{itemize}
		\end{theorem}
		Consequently, decomposing $\ShST(\lambda/\mu,n)$ into its connected components yields an isomorphism of crystals
		\[\ShST(\lambda/\mu,n)\ \cong\ \bigsqcup_\nu \ShST(\nu,n)^{f_{\nu,\mu}^\lambda},\]
		where $f_{\nu,\mu}^\lambda$ is the coefficient of the Schur Q function $Q_\nu$ in the expansion of the skew Schur Q function $Q_{\lambda/\mu}$. Taking generating functions (weighting a vertex of weight $\gamma$ by $2^{\#\{i:\gamma_i\neq 0\}}$) recovers the skew type B Littlewood-Richardson rule,
		\[Q_{\lambda/\mu}(x_1, \ldots, x_n) = \sum_\nu f_{\nu,\mu}^\lambda Q_\nu(x_1, \ldots, x_n).\]
		The crystal structure also gives an automatic proof of symmetry for $Q_\lambda$ (Corollary \ref{cor:symmetry-schurQ}).
		
		Despite these connections to the type B Littlewood-Richardson rule, our crystal is not a type B crystal. Instead, it has the following `doubled' type A structure, based on considering the primed and unprimed operators considered separately.
		\begin{theorem}\label{thm:main-doubled-typeA}
			The operators $F_i,E_i,F'_i,E'_i$ commute whenever compositions are defined. Moreover, $F_i,E_i$ and $F'_i,E'_i$  independently satisfy the type A Kashiwara crystal axioms, using the same auxiliary functions $\varepsilon_i,\varphi_i,\wt$  on $\ShST(\lambda/\mu,n)$.
		\end{theorem}
		
		\begin{remark}
			The two Kashiwara crystals generated by the $F_i,E_i$ or the $F_i',E_i'$ operators do not satisfy the Stembridge axioms for type A, and therefore are not crystals in the sense of the representation theory of $U_q(\mathfrak{sl}_n)$. In particular, they are not \textbf{seminormal}: the auxiliary functions $\varepsilon_i$ and $\varphi_i$ do not measure the distances to the ends of an $F_i$-chain, but rather the \textit{total} distance to the end of a (possibly two-row) $F_i$/$F_i'$ chain, as in Figure \ref{fig:two-row}.
		\end{remark}
		
		Finally, we prove a uniqueness statement, that our graphs are combinatorially uniquely-determined by the relations satisfied by the $i',i,j',j$ operators, particularly $j=1+1$. To be precise, let $G$ be a finite $\mathbb{Z}^n$-weighted graph with edges labeled $i',i$ for $i=1, \ldots, n-1$. For any $i,j$, let $G^{i,j}$ be the subgraph obtained by deleting all but the $i',i,j',j$ edges.
		
		\begin{theorem}
			\label{thm:uniqueness-main}
			Suppose $G$ is connected and satisfies the following:
			\begin{itemize}
				\item For each $i$, and each connected component $C \subset G^{i,i+1}$, there is a strict partition $\lambda$ such that $C \cong \ShST(\lambda,3)$ (using the $i,i+1,i+2$ parts of the weight function).
				\item For $|i-j|>1$, for each connected component $C \subset G^{i,j}$, there are strict partitions $\lambda, \mu$ such that $C \cong \ShST(\lambda,2) \times \ShST(\mu,2)$ (using the $i,j$ parts of the weight function).
			\end{itemize}
			Then $G$ has a unique highest-weight element $g^*$ and $G \cong \ShST(\lambda,n)$, where $\lambda = \mathrm{wt}(g^*)$.
		\end{theorem}
		Thus, any graph that is locally (shifted-)crystal-like is globally a shifted tableau crystal. In particular, the generating function of such a graph (where we weight a vertex of weight $\gamma$ by $2^{\#\{i:\gamma_i\neq 0\}}$) is Schur-Q-positive. The analogous statement for ordinary tableaux crystals is due to Stembridge \cite{Stembridge-typeA}.
		
		\subsubsection{Lattice walks and critical strings}
		
		The key construction underlying the definitions of $E_i, F_i$ is to associate, to each word $w$ in the alphabet $\{i',i,i+1',i+1\}$, a first-quadrant {\bf lattice walk}, beginning at the origin. See Figure \ref{fig:lattice-walk-intro} for an example. This walk determines the rectification shape of $w$:
		\begin{theorem}
			The lattice walk of a word $w$ in the alphabet $\{i',i,i+1',i+1\}$ determines the shape and weight of $\rectify(w)$.
		\end{theorem}
		In fact more is true: the walk almost determines the shifted dual equivalence class of $w$ (precise statements are given in Corollary \ref{cor:initial-subwords-dual} and Lemma \ref{lem:circling}). As an additional corollary, we obtain a new criterion for ballotness, which differs from existing characterizations (see \cite{Stembridge}) in that it only requires reading through $w$ once, rather than twice (backwards-and-forwards).
		\begin{theorem} \label{thm:lattice-walk-main}
			Let $w$ be a word in the alphabet $\{1',1, \ldots, n',n\}$. Then $w$ is ballot if and only if each of its lattice walks (for $i=1, \ldots, n-1$) ends on the $x$-axis; it is anti-ballot if and only if each lattice walk ends on the $y$-axis.
		\end{theorem}
		
		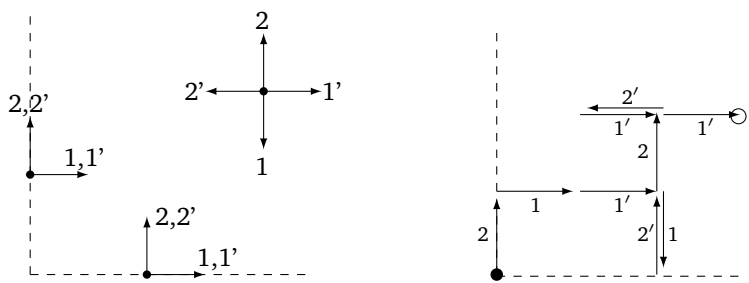
\begin{figure}
			\begin{center}
				\begin{tabular}{cc}
					\setlength{\unitlength}{3em}
					\ \ \begin{picture}(5,3)(0,0)
					\multiput(0,0)(0,0.2){16}{\line(0,1){0.1}}
					\multiput(0,0)(0.2,0){17}{\line(1,0){0.1}}
					\put(0,1.2){\circle*{0.1}}
					\put(0,1.2){\vector(0,1){.7}\vector(1,0){.7}}
					\put(-0.265,1.95){\small{2,2'}}
					\put(0.4,1.3){\small{1,1'}}%
					\put(1.4,0){\circle*{0.1}}
					\put(1.4,0){\vector(0,1){.7}\vector(1,0){.7}}
					\put(1.5,0.6){\small{2,2'}}
					\put(2.0,0.1){\small{1,1'}}%
					\put(2.8,2.2){\circle*{0.1}}
					\put(2.8,2.2){\vector(0,1){.7}\vector(1,0){.7}}
					\put(2.8,2.2){\vector(0,-1){.7}}
					\put(2.8,2.2){\vector(-1,0){.7}}
					\put(2.7,2.95){\small{2}}
					\put(3.5,2.1){\small{1'}}
					\put(2.7,1.2){\small{1}}
					\put(1.85,2.1){\small{2'}}
					\end{picture}
					&
					\ \
					\begin{picture}(3,3)(0,0)
					\setlength{\unitlength}{3em}
					\stepnorth{2}{%
						\stepeast{1}{%
							\stepeast{1'}{%
								\stepsouth{1}{%
									\stepnorthshift{2'}{%
										\stepnorthshift{2}{%
											\stepwest{2'}{%
												\stepeastshift{1'}{%
													\stepeastshift{1'}{%
													}}}}}}}}}
													\put(0,0){\circle*{0.15}}
													\put(2.9,1.9){\circle{0.18}}
													\multiput(0,0)(0,0.2){15}{\line(0,1){0.1}}
													\multiput(0,-0.02)(0.2,0){15}{\line(1,0){0.1}}
													\end{picture}
												\end{tabular}
											\end{center}
											\caption{The {\bf lattice walk} of a word $w = w_1 w_2 \cdots w_n \in \{1', 1, 2', 2\}^n$. In the interior of the first quadrant, each letter corresponds to a cardinal direction. Along the axes, primed and unprimed letters behave the same way. \textbf{Right:} The walk for $w=211'12'22'1'1'$ ends at the point $(3,2)$.\label{fig:lattice-walk-intro}} \end{figure}
										
										The lattice walk is similar to the bracketing rule for the ordinary crystal operators: the arrows that occur far from the $x$- or $y$-axis `cancel' in opposite pairs, and the transformation is done on the remaining subword. In particular, the operators $E_i,F_i$ are defined in terms of transforming {\bf critical substrings} of $w$: certain specific types of substring, which are required to occur when the lattice walk passes close to either the $x$- or $y$-axis.
										
										%indicating that the corresponding prefix of $w$ is nearly ballot or anti-ballot.
										
										\subsection{Applications and future work}
										
										As one immediate application, in \cite{GLP}, we use this structure to understand the topology of so-called Schubert curves in the real odd orthogonal Grassmannian $\mathrm{OG}(n+1,2n+1)$. In prior work (\cite{GillespieLevinson}), a local algorithm was developed for computing the real topology of Schubert curves in the ordinary Grassmannian $\mathrm{Gr}(k,n)$, based on a covering map to the circle $\mathbb{RP}^1$. The operators defining this algorithm are by their nature coplactic operators on skew semistandard tableaux, and can be written in terms of the usual crystal operators on ordinary tableaux.
										
										The development of the operators $F_i,F'_i,E_i,E'_i$ in this paper was originally motivated by the analogous geometric question in the orthogonal Grassmannian, and our operators indeed give rise to a combinatorial algorithm for computing the monodromy of the analogous covering map in type B. We believe, however, that this new combinatorial structure on shifted tableaux will also be useful more broadly.
										
										A significant question is whether our crystals form canonical bases for the representations of some quantized enveloping algebra.  For instance, the quantum queer superalgebra \cite{GJKKK} has associated crystal bases whose elements are related to shifted tableaux, and a possible link between our crystals and this or other algebras would be worth pursuing. It would also be interesting to define tensor products on an appropriate category of `doubled crystals'. We note also that our ballotness criterion (Theorem \ref{thm:lattice-walk-main}) is simpler than existing characterizations, and is well-suited for tensor-product-like constructions such as concatenations of words (see e.g.~Corollary \ref{cor:concat-ballot}).
										
										Our crystals also recover the combinatorics of the Schur $Q$-functions, giving new proofs of their symmetry and of the type B Littlewood-Richardson rule. Our uniqueness result also yields a crystal-theoretic way to find the explicit expansion of a Schur $Q$-positive symmetric function in terms of Schur $Q$-functions: one introduces operators on the underlying set, satisfying the appropriate local relations.  Theorem \ref{thm:uniqueness-main} then proves that the resulting generating function is Schur-Q-positive.  The analogous method in type A has been applied successfully in \cite{Morse-Schilling} for certain affine Stanley symmetric functions. For this purpose, it would also be interesting to find more explicit axioms, similar to Stembridge's \cite{Stembridge-typeA} for crystals of simply-laced root systems.  
										
										%One unresolved mystery is that our crystal structure is closest to that of type A crystals since the weights are in the type A weight lattice, even though the geometric motivation was from type B.  If a Stembridge-like axiomatic characterization is developed for these crystals, it may be worth exploring the analogous axioms for other simply-laced types and obtaining other types of 'doubled crystals'.
										
										\subsection{Structure of the paper}
										The paper is organized as follows.  We set notation in Section \ref{sec:notation}, then in Section \ref{sec:primed-ops} we introduce the primed operators $E_i',F_i'$ and prove that they are coplactic. In Section \ref{sec:walks}, we introduce the lattice walk and prove Theorem \ref{thm:lattice-walk-main}. In Section \ref{sec:unprimed-ops}, we define the unprimed operators $E_i,F_i$ on words, show that they are well-defined on tableaux and are coplactic, and complete the proof of Theorem \ref{thm:main}. In Section \ref{sec:crystal}, we study the joint crystal structure determined by the primed and unprimed operators and prove Theorem \ref{thm:main-doubled-typeA}. In Section \ref{sec:characters}, we prove the statements about the Schur $Q$-functions. Finally, in Section \ref{sec:uniqueness}, we prove the uniqueness statement (Theorem \ref{thm:unique-iso}).
										
										\subsection{Acknowledgments}
										
										We thank Anne Schilling, David Speyer, John Stembridge and Mark Haiman for helpful conversations pertaining to this work. Computations in Sage \cite{sage} were also useful for analyzing the combinatorics of the operators.
										
										\section{Background and Notation}\label{sec:notation}
										
										\subsection{Strings and words}
										
										Let $w = w_1w_2 \dots w_n$ be a string in symbols $\{1',1,2',2,3',3,\ldots\}$. We will normally assume that the first occurrence of $\{i,i'\}$ in $w$ is $i$, for each $i\in \{1,2,3,\ldots\}$.  Informally, we adopt the convention that the first $i$ can be treated as $i'$ for the purposes of any rule we state, and if any operation produces a string where the first symbol among $\{i,i'\}$ is $i'$, then we will implicitly change it to $i$. We make this rigorous as follows.
										
										\begin{definition}
											Let $w$ be a string in symbols $\{1',1,2',2,3',3,\ldots\}$.  The \textit{first $i$ or $i'$} of $w$ is the leftmost entry which is either equal to $i$ or $i'$.  The \textbf{canonical form} of $w$ is the string formed by replacing the first $i$ or $i'$ (if it exists) with $i$ for all $i\in \{1,2,3,\ldots\}$.  We say two strings $w$ and $v$ are \textbf{equivalent} if they have the same canonical form; note that this is an equivalence relation.
										\end{definition}
										
										\begin{definition}
											A \textbf{word} is an equivalence class $\hat{w}$ of the strings $v$ equivalent to $w$.  We have $w\in \hat{w}$, and we say that $w$ is the \textbf{canonical representative} of the word $\hat{w}$.  We often call the other words in $\hat{w}$ \textbf{representatives} of $\hat{w}$ or of $w$.  The {\bf weight} of $w$ is the vector $\mathrm{wt}(w) = (n_1, n_2, \ldots ),$ where $n_i$ is the total number of $(i)$s and $(i')$s in $w$.
										\end{definition}
										
										\begin{example}
											The canonical form of the word $1'1'2'112'$ is $11'2112'$.  The set of all representatives of $11'2112'$ is $\{1'1'2'112',11'2'112',1'1'2112',11'2112'\}$.  Note that the word $11'1$ only has two distinct representatives instead of four.
										\end{example}
										
										To make rigorous the notion of a partial operator on words, it is often helpful to start with partial operators on strings and take the natural induced operator on words.  By a \textbf{partial operator} on a set $S$ we mean an operator $A:T\to S$ defined on some subset $T\subseteq S$, and we write $A(s)=\varnothing$ when $A$ is not defined on $s$ (i.e., $s\not\in T$).
										
										\begin{definition}
											Let $A$ be a partial operator on the set of finite strings whose elements are in the alphabet $\{1,1',2,2',3,3',\ldots\}$.  We say that $A$ is \textbf{defined on a word $w$} if it is defined on some representative $v$ of $w$ and all such representatives give outputs $A(v)$ that represent the same word.  In that case we define the \textbf{induced operator} $\hat{A}$ on words by $\hat{A}(w)=\widehat{A(v)}$ if such a $v$ exists, and $\hat{A}(w)=\varnothing$ otherwise.
										\end{definition}
										
										\begin{remark}
											We will often abuse notation by referring to a word $\hat{w}$ by its canonical element $w$, and by referring to $A$ in place of $\hat{A}$ throughout.
										\end{remark}
										
										\subsection{Shifted tableaux and jeu de taquin}
										
										Recall that a \textbf{strict partition} is a strictly-decreasing sequence of positive integers, $\lambda=(\lambda_1 > \ldots > \lambda_k)$.  We say that $|\lambda|=\sum \lambda_i$ is the \textbf{size} of $\lambda$, and the entries $\lambda_i$ are the \textbf{parts} of $\lambda$.  The \textbf{shifted Young diagram} of $\lambda$ is the partial grid of squares in which the $i$-th row contains $\lambda_i$ boxes and is shifted to the right $i$ steps, as in the example shown below.  A \textbf{(shifted) skew shape} is a difference $\lambda/\mu$ of two partition diagrams, formed by removing the squares of $\mu$ from the diagram of $\lambda$, if $\mu$ is contained in $\lambda$.  
										
										\begin{center}
											\includegraphics{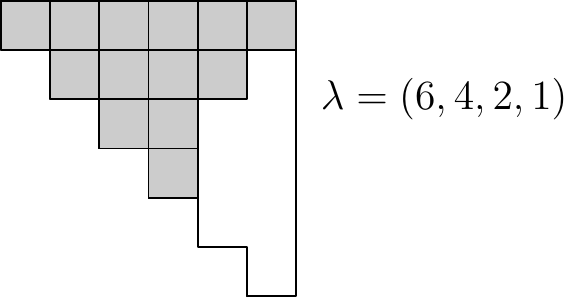} \hspace{1cm}
											\includegraphics{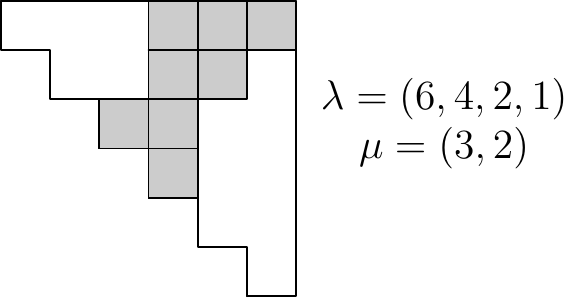}
										\end{center}
										
										A \textbf{shifted semistandard Young tableau} is a filling of the boxes with entries from the alphabet $\{1'<1<2'<2<3'<3<\cdots \}$ such that the entries are weakly increasing down columns and across rows, and such that primed entries can only repeat in columns, and unprimed only in rows.  The \textbf{(row) reading word} of such a tableau is the word formed by concatenating the rows from bottom to top (in the example below, the reading word is $3111'21'12'$).  We also require that the first $i$ or $i'$ in reading word order is always unprimed. The {\bf weight} of $T$ is the vector $\mathrm{wt}(T) = (n_1, n_2, \ldots)$, where $n_i$ is the total number of $(i)$s and $(i')$s in $T$.
										
										\begin{center}
											\includegraphics{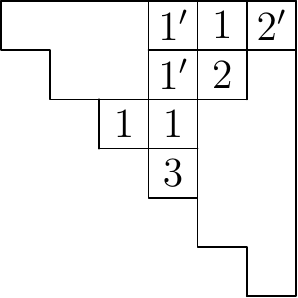}
										\end{center}
										
										The notion of \textbf{jeu de taquin} for shifted tableaux is the same as for usual tableaux: an \textbf{inner jeu de taquin slide} is the process of choosing an empty inner corner of the skew tableau and choosing either the entry to its right or below it to slide into the empty square so as to keep the tableau semistandard, then repeating the process with the new empty square, and so on until the empty square is an outer corner. An \textbf{outer jeu de taquin slide} is the reverse process, starting with an outer corner and sliding boxes outwards.
										
										There is one exception to the sliding rules: if an outer slide moves an $i$ down into the diagonal and then another $i$ to the right on top of it, that $i$ becomes primed (and vice versa for the corresponding inner slide), as shown below.
										\begin{center}
											\includegraphics{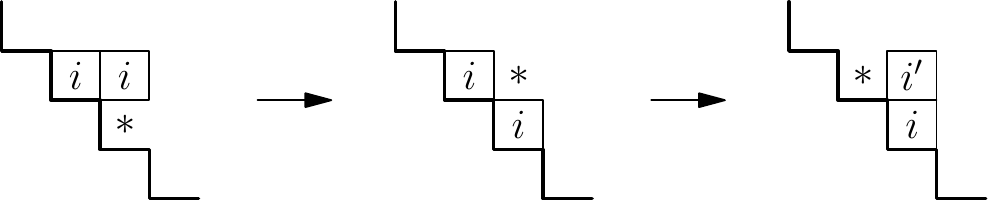}
										\end{center}
										
										We write $\rectify(T)$ or $\rectify(w)$ to denote the jeu de taquin \textbf{rectification} of any shifted semistandard tableau $T$ with reading word $w$; this is well-defined by \cite{Sagan}, \cite{Worley}. We say that $T$ is {\bf Littlewood-Richardson}, and $w$ is {\bf ballot}, if for every $i$, the $i$-th row of $\rectify(T)$ consists entirely of $(i)$s.
										
										\begin{definition}
											Let $T$ be a semistandard shifted skew tableau with reading word $w$.  We say that $T$ is in \textbf{canonical form} if $w$ is, and use the same conventions for tableaux as for words: a tableau $T$ in canonical form is identified with the set of \textbf{representatives} of $T$, tableaux formed by possibly priming the first $i$ in reading order for each $i$. For jeu de taquin, this introduces a second `special slide' by priming the first $i$ in reading order:
											\begin{center}
												\includegraphics{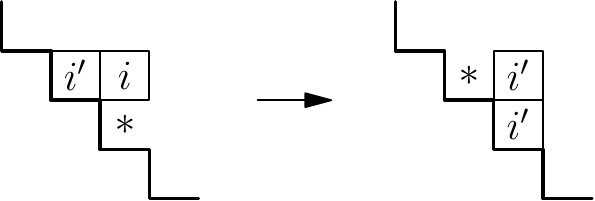}
											\end{center}
											In particular, switching representatives commutes with jeu de taquin, so it always suffices to do jeu de taquin on tableaux in canonical form.
										\end{definition}
										
										The main property that we wish our operators $E_i$, $E'_i$, $F_i$, $F'_i$ to satisfy is \textit{coplacticity}.
										
										\begin{definition}
											An operation on canonical shifted tableaux (or on their reading words) is \textbf{coplactic} if it commutes with all shifted jeu de taquin slides.
										\end{definition}
										
										\begin{definition}
											The \textbf{standardization} of a word $w$ is the word $\mathrm{std}(w)$ formed by replacing the letters in order with $1,2,\ldots,n$ from least to greatest, breaking ties by reading order for unprimed entries and by reverse reading order for primed entries.  
											
											The standardization of a shifted tableau $T$ is the standard shifted tableau $\std(T)$ formed by standardizing its reading word and performing the same changes on the corresponding letters of the tableau.
										\end{definition}
										
										\begin{example}
											The standardization of the word $1121'22'1'11$ is the word 
											$348297156$.  Note that this is the same as the standardization of every representative of the word, and in general standardization is well-defined independently of the representative.
										\end{example}
										
										\begin{proposition}\label{prop:standardization}
											A filling $T$ of a shifted skew shape is semistandard if and only if it is in canonical form and $\std(T)$ is a standard shifted tableau.
										\end{proposition}
										
										\begin{proof}
											The forward direction is clear.  For the reverse direction, suppose $T$ is in canonical form and $\std(T)$ is a standard shifted tableau.  we need to check that any two adjacent entries satisfy the semistandard condition.  If $x$ is just left of $y$, then $x<y$ in $\std(T)$.  Thus either $x<y$ in $T$ or $x=y$ is unprimed.  Similarly, if $x$ is just above $y$ then either $x<y$ in $T$ or $x=y$ is primed.  Thus $T$ is semistandard (since it is also canonical.)
										\end{proof}
										
										\subsubsection{Shifted dual equivalence}
										
										Two \textit{standard} skew shifted tableaux are said to be (shifted) \textbf{dual equivalent} if their shapes transform the same way under any sequence of jeu de taquin slides. (See \cite{Assaf}, \cite{Haiman} for more in-depth discussions of dual equivalence.)  We extend this notion to \textit{semistandard} shifted tableaux by defining two tableaux to be \textbf{dual equivalent} if and only if their standardizations are, as in \cite{SaganBook}. The word `dual' refers to the recording tableau under the shifted Schensted correspondence \cite{Sagan, Worley}. 
										
										We will occasionally think of a word $w = w_1 \cdots w_n$ as a diagonally-shaped tableau (of skew shape $(2n-1, \ldots,3, 1) / (2n-3, \ldots, 1)$.) We will say that two words $w, w'$ are dual equivalent if the corresponding diagonal tableaux are dual equivalent.
										
										\subsection{The involutions $\eta_i$}
										
										We will also make use of an important symmetry on words and tableaux, as follows.
										
										\begin{definition}
											For a word $w$, we define $\eta_i(w)$ by replacing the letters $\{i',i,i+1', i+1\}$ of any representative of $w$ as follows: we send $i' \mapsto i+1$, $i \mapsto i+1'$, $i+1' \mapsto i$, and $i+1 \mapsto i'$.  (Note that the choice of representative $v$ of the word $w$ does not affect the output.)
										\end{definition}
										
										\begin{example} 
											We have $\eta_1(121'132) = 2122'31'$.
										\end{example}
										
										It will also be helpful to extend $\eta_i$ to tableaux having only $i',i,i+1',i+1$ entries.  For simplicity we consider $\eta_1$ and the letters $1$ and $2$.
										
										\newcommand{\etaT}{\boldsymbol{\eta}}
										\begin{definition}\label{def:eta-on-T}
											Let $T$ be a skew shifted tableau in the $n\times n$ staircase with entries in $\{1',1,2',2\}$.  Define $\etaT_1(T)$ to be the tableau formed by applying $\eta_1$ to the entries of $T$, reflecting the tableau across the antidiagonal of the $n\times n$ staircase, and reverting to canonical form.
										\end{definition}
										
										\begin{center}
											\raisebox{-.5\height}{\includegraphics[scale=0.8]{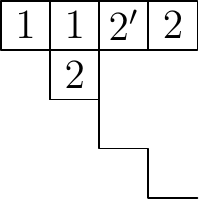}} \qquad $\xrightarrow{\ \eta_1\ }$ \qquad \raisebox{-.5\height}{\includegraphics[scale=0.8]{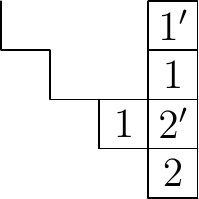}}
										\end{center}
										
										To relate this definition back to $\eta_1$ on words, we also define the \textbf{column reading word} of a tableau to be the word formed by concatenating the columns from left to right, reading each column from bottom to top.
										
										\begin{remark} \label{rmk:row-col-eta}
											It is easy to see that if $w$ is the (row) reading word of $T$ in Definition \ref{def:eta-on-T} and $v$ is the column reading word of $\etaT_1(T)$, then $v=\eta_1(w)$.
										\end{remark}
										
										\begin{remark} \label{rmk:eta-words-dual}
											By construction, $\etaT$ is a coplactic operation on tableaux. In particular, two tableaux $T$ and $T'$ are dual equivalent if and only if $\etaT_1(T)$ and $\etaT_1(T')$ are dual equivalent. For words, the situation is similar: $w$ and $w'$ are dual equivalent if and only if $\eta_1(w)$ and $\eta_1(w')$ are dual equivalent (noting that $\etaT_1$ and $\eta_1$ have the same effect on diagonally-shaped tableaux).
										\end{remark}

										%%%%%%%%%%%%%%%%
										
										\section{The operators $E'$ and $F'$}\label{sec:primed-ops}
										
										Throughout this section, we consider words consisting only of the letters $\{1',1,2',2\}$, and define only the operators $E'_1$ and $F'_1$.  For general words $w$, $E'_i$ and $F'_i$ are defined on the subword containing the letters $\{i',i,i+1',i+1\}$, treating $i$ as $1$ and $i+1$ as $2$.
										
										For simplicity we use the following shorthands.
										
										\begin{definition}
											Define $E' = E'_1$, $F' = F'_1$, $\eta= \eta_1.$
										\end{definition}
										
										To begin, we make the following observation.
										
										\begin{lemma} \label{lem:unstandardize}
											Let $s$ be a standard word (i.e. a permutation of $1, \ldots, n$), and let $n = \sum_{i=1}^k a_i$. There is at most one word $w$ of weight $(a_1,\ldots,a_k)$ and standardization $\mathrm{std}(w) = s$.
										\end{lemma}
										
										\begin{proof}
											We construct $w$ from $s$. The numbers $1, \ldots, a_1$ of $s$ must become $(1)$s and $(1')$s. Moreover, the assignment of primes is uniquely determined: if $i$ is the first in reading order among $1, \ldots, a_1$, then $1, \ldots, i-1$ must become $(1')$s, and must occur in reverse reading order (or else no such $w$ exists). Similarly, $i, \ldots, a_1$ must become $(1)$s and must occur in reading order. Next, $a_1+1, \ldots, a_1+a_2$ must become $(2)$s and $(2')$s, and the assignment of primes is again (at most) uniquely determined. The argument continues inductively.
										\end{proof}
										
										Consequently, the following definition makes sense. Let $\alpha$ be the vector $(1,-1)$ and let $w$ be a word.
										
										\begin{definition}[Primed operators] \label{def:primed-operators}
											We define $E'(w)$ to be the unique word such that
											\[\mathrm{std}(E'(w)) = \mathrm{std}(w) \hspace{0.5cm}\text{ and }\hspace{0.5cm} \mathrm{wt}(E'(w)) = \mathrm{wt}(w) + \alpha,\]
											if such a word exists; otherwise, $E'(w) = \varnothing$. We define $F'(w)$ analogously using $-\alpha$.
										\end{definition}
										
										The key properties of $E'$ and $F'$ are as follows.
										
										\begin{proposition} \label{prop:primed-partial-inverses}
											The maps $E'$ and $F'$ are partial inverses of each other, that is, $E'(w) = v$ if and only if $w = F'(v)$.
										\end{proposition}
										
										\begin{proof}
											This is immediate from the definitions of $E'$ and $F'$.
										\end{proof}
										
										\begin{proposition} \label{prop:primed-eta}
											We have $E' = \eta \circ F' \circ \eta$.
										\end{proposition}
										
										\begin{proof}
											Observe that $\eta$ inverts the standardization numbering, sending $i \mapsto n+1-i$, and so $\mathrm{std}(\eta \circ F' \circ \eta(w)) = \mathrm{std}(w)$. Then observe that $\eta$ also reverses the weight.
										\end{proof}
										
										\begin{proposition} \label{prop:primed-semistandard}
											The maps $E'$ and $F'$ are well-defined on skew shifted semistandard tableaux, that is, they preserve semistandardness.
										\end{proposition}
										
										\begin{proof}
											Since $E'$ and $F'$ preserve standardization, this follows from Proposition \ref{prop:standardization}. 
										\end{proof}
										
										\begin{proposition} \label{prop:primed-coplactic}
											The operations $E'$ and $F'$ are coplactic, that is, they commute with all jeu de taquin slides.
										\end{proposition}
										
										\begin{proof}
											The slide path of a jeu de taquin slide on $T$ depends only on $\mathrm{std}(T)$. Hence, if $T' = E'(T)$ or $F'(T)$, applying a jeu de taquin slide gives a new pair of tableaux with the same standardization, and with weight differing by $\pm \alpha$, as desired.
										\end{proof}
										
										We also note that the primed operators $F'_i, E'_i$ for \emph{different} $i$ commute when defined:
										
										\begin{proposition} \label{prop:primed-operators-commute}
											Let $w$ be a word on an arbitrary alphabet $\{1', 1, \ldots, n', n\}$. Then $F'_i F'_j(w) = F'_j F'_i(w)$ whenever both are defined, and likewise for all other pairs from $F'_i, E'_i, F'_j, E'_j$.
										\end{proposition}
										
										\begin{proof}
											This is immediate from the definition, since the effect is to change $\mathrm{wt}(w)$ by the corresponding sum of vectors $\pm \alpha_i$, while preserving $\mathrm{std}(w)$.
										\end{proof}

										We now describe the action of $E'$ and $F'$ in a computationally convenient way.
										
										\begin{proposition}\label{prop:explicit-definitions-primed}
											To compute $F'(w)$, if the last $1$ in $w$ can be changed to $2'$ without affecting $\std(w)$, then $F'(w)$ is formed by making that change (and canonicalizing).  Otherwise $F'(w)=\varnothing$.
											
											For $E'(w)$, let $x$ be the last $2'$ in $w$ and let $y$ be the first $1$.  If $x$ is left of $y$ and $y$ is the only $1$ in $w$, then $E'(w)$ is obtained by changing $x$ to $1$ and $y$ to $1'$.  If simply changing $x$ to $1$ does not affect $\std(w)$, then $E'(w)$ is formed by making that change and canonicalizing.  Otherwise $E'(w)=\varnothing$.
										\end{proposition}
										
										\begin{proof}
											Since $F'$ is the unique operator that preserves standardization and lowers the weight, say from $(a,b)$ to $(a-1,b+1)$, we see from the construction in the proof of Lemma \ref{lem:unstandardize} that the entry to change from a $1$ or $1'$ to a $2$ or $2'$ is precisely the entry that standardizes to $a$.  This is the last $1$ in reading order, and changing it to a $2'$ is the only potential way to preserve standardization.
											
											The analysis is similar for $E'$, with the exception that if the first $1$, which can act as a $1'$, is the only $1$ and the last $2'$ lowers to a $1$ to its left, we can also prime this $1$ and lower the $2'$ to preserve standardization. 
										\end{proof}
										
										\begin{remark}
											The computational definitions for $E'$ and $F'$ are asymmetric because of our choice of unprimed letters for canonical form.
										\end{remark}
										
										This also allows us to describe the action on straight shifted tableaux $T$.
										
										\begin{proposition} \label{prop:rect-primed-action}
											If $T$ has one row, $E'(T)$ (respectively $F'(T)$) is obtained by changing the leftmost $2$ to a $1$ (respectively, the rightmost $1$ to a $2$), if possible; otherwise it is $\varnothing$.
											
											If $T$ has two rows and its first row contains a $2'$, then $E'(T)$ is obtained by changing the $2'$ to a $1$, and $F'(T) = \varnothing$. If the first row does not contain a $2'$, $E'(T) = \varnothing$ and $F'(T)$ is obtained by changing the rightmost $1$ to a $2'$.
										\end{proposition}
										
										\begin{example} Here are some maximal chains for $F'$:
											\begin{gather*}
												12211' \xrightarrow{F'} 1222'1' \xrightarrow{F'} \varnothing
												\\
												1111'1' 
												\xrightarrow{F'} 1121'1' 
												\xrightarrow{F'} 1221'1' 
												\xrightarrow{F'} 22211' 
												\xrightarrow{F'} 2222'1
												\xrightarrow{F'} 2222'2'
												\xrightarrow{F'} \varnothing
											\end{gather*}
										\end{example}

										We conclude this section with a note on the lengths of chains in the primed operators. The maximal chains for $F'$ will normally have length $1$ (as in the first example above).  We get longer chains (as in the second example above) if and only if $\rectify(w)$ has only one row.  
										
										\begin{proposition}\label{lem:maxchains}
											A maximal chain for $F'$ has length greater than $1$ if and only if $\rectify(w)$ has one row (and at least two boxes). Moreover, when $\rectify(w)$ has two rows, $F'(w) = \varnothing$ iff $E'(w) \neq \varnothing$ iff $\rectify(w)$ contains a $2'$ in its first row.
										\end{proposition}
										
										\begin{proof}
											By coplacticity, it suffices to show this for the reading word of a straight shifted tableau. Then we apply the previous Proposition. \end{proof}

										\section{The lattice walk of a word}\label{sec:walks}
										
										The first step in defining the more involved operators $F(w)$ and $E(w)$ is to associate, to each word $w$, a lattice walk in the first quadrant of the plane.  This walk is a sequence $$P_0(w), P_1(w), \ldots, P_n(w)$$ of points in $\NN \times \NN$, starting with $P_0(w) = (0,0)$.  We specify the walk by assigning a step to each $w_i$, $i=1, \dots, n$. This step will be one of the four principal direction vectors:
										\[
										\east{~~} \ =\  (1,0)
										\qquad \west{~~} \ =\  (-1,0)
										\qquad \north{} \ =\  (0,1)
										\qquad \south{} \ =\  (0,-1)\,.
										\]
										$P_i(w)$ is (as usual) the sum of the steps assigned to $w_1, \dots, w_i$. We define the walk inductively, as follows.  Suppose $i >0$, and we have assigned steps to $w_1, \dots, w_{i-1}$, so $P_0(w), \dots,P_{i-1}(w)$ have been defined.  Write $P_{i-1}(w) = (x_{i-1},y_{i-1})$.  We assign the step to $w_i$ according to
										Figure \ref{fig:directions}, with two cases based on whether or not the step starts on one of the $x$ or $y$ axes. See Figure \ref{fig:lattice-walk-intro}. We will generally write the label each step of the walk by the letter $w_i$, so as to represent both the word and its walk on the same diagram.

										\begin{figure}
											\begin{center}
												\begin{tabular}{|c|cccc|}
													\hline
													$x_iy_i=0$ & \east{1'} & \east{1} & \north{2'} & \north{2}   \\[.8ex]
													\hline
													$x_iy_i\neq0$ &\east{1'} & \south{1} & \west{2'} & \north{2}   \\[.8ex]
													\hline
												\end{tabular}
											\end{center}
											\caption{\label{fig:directions} The directions assigned to each of the letters $w_i=1'$, $1$, $2'$, or $2$ according to whether the location of the walk just before $w_i$ starts on the axes ($x_iy_i=0$) or not.}
										\end{figure}
										
										\begin{example}
											\label{ex:walk}
											Here is the walk for $w = 1221'1'111'1'2'2222'2'11'1$.
											
											\begin{center}
												\begin{picture}(5,4.2)(0,0)
												\multiput(0,0)(0,0.2){22}{\line(0,1){0.1}}
												\multiput(0,0)(0.2,0){27}{\line(1,0){0.1}}
												\put(0,0){\circle*{0.13}}
												\put(4,2){\circle{0.13}}
												\stepeast{1}{%
													\stepnorth{2}{%
														\stepnorth{2}{%
															\stepeast{1'}{%
																\stepeast{1'}{%
																	\stepsouth{1}{%
																		\stepsouth{1}{%
																			\stepeast{1'}{%
																				\stepeast{1'}{%
																					\stepnorth{2'}{%
																						\stepnorth{2}{%
																							\stepnorth{2}{%
																								\stepnorth{2}{%
																									\stepwest{2'}{%
																										\stepwest{2'}{%
																											\stepsouth{1}{%
																												\stepeast{1'}{%
																													\stepsouth{1}{%
																													}}}}}}}}}}}}}}}}}}
																													\end{picture}
																												\end{center}
																											\end{example}
																											
																											\begin{proposition}
																												The walk for $\eta(w)$ is the reflection of the walk for $w$ over the line $y=x$.
																											\end{proposition}
																											
																											\begin{proof}
																												This is clear by the symmetry of the lattice walk operations.
																											\end{proof}
																											
																											The key property of this lattice walk is that its length, $n$, and its endpoint, $P_n(w) = (x_n, y_n)$, tell us almost everything about the rectification tableau $\rectify(w)$. Note that $\rectify(w)$ is a straight shifted tableau with at most two rows.
																											
																											\begin{theorem}\label{thm:rectification}
																												The shape of $\rectify(w)$ is $\lambda = (\lambda_1, \lambda_2)$, where
																												\begin{align*}
																													\lambda_1 &= \tfrac{1}{2}(n+x_n+y_n) = \#\big\{\ \north{} \text{ and/or } \east{} \text{ steps in the walk\ }\big\}, \\
																													\lambda_2 &= \tfrac{1}{2}(n-x_n-y_n) = \#\big\{\west{} \text{ and/or } \south{} \text{ steps in the walk\ }\big\},
																												\end{align*}
																												Moreover, there are $\tfrac{1}{2}(n+x_n-y_n)$ $(1)$s in the first row; there is at most one $2'$, and the remaining entries in the tableau are all $(2)$s.
																											\end{theorem}
																											
																											\begin{corollary}
																												The shape of $\rectify(w)$ has only one row iff all
																												steps in the walk are $\east{~}$ or $\north{}$.
																											\end{corollary}
																											
																											As an additional corollary, we obtain a new criterion for ballotness:
																											
																											\begin{corollary} \label{cor:ballot-walk-criterion}
																												A word $w$ is ballot if and only if for all $i$, the lattice walk of the subword $w_i$ consisting of the letters $i,i',i+1,i+1'$ has $y_n = 0$, that is, ends on the $x$-axis. Similarly $w$ is opposite ballot (i.e. $\eta(w)$ is a ballot word) if and only if $x_n = 0$ for all $i$.
																											\end{corollary}
																											
																											\begin{proof}
																												It is well-known that a word is ballot if and only if each subword $w_i$ is ballot (see \cite{Stembridge}, for instance).  By definition, $w_i$ is ballot if and only if the first row of $\rectify(w_i)$ consists only of $(1)$s (after replacing $i',i \leadsto 1',1$ and $i+1',(i+1) \leadsto 2',2$). By Theorem \ref{thm:rectification}, this says $\tfrac{1}{2}(n+x_n-y_n) = \tfrac{1}{2}(n+x_n+y_n)$, that is, $y_n=0$.
																											\end{proof}

																											Our proof uses the rules of shifted Knuth equivalence, as developed in \cite{Sagan} and \cite{Worley}.  Using our primed notation, their shifted Knuth moves are as follows.
																											
																											\begin{definition}
																												Two words $w$ and $v$ are \textbf{shifted Knuth equivalent} if they differ by a sequence of elementary shifted Knuth moves on adjacent letters, consisting of either:
																												\begin{itemize}
																													\item $bac\leftrightarrow bca$ where under the standardization ordering we have $a<b<c$,
																													\item $acb\leftrightarrow cab$ where under the standardization ordering we have $a<b<c$,
																													\item Interchanging the first two letters at the start of the word,
																													\item If the first two letters in the word are unprimed and equal, priming the second; or, if the first two letters are $aa'$, un-priming the second.
																												\end{itemize}
																												In the moves $bac \leftrightarrow bca$ and $acb \leftrightarrow cab$, we call $b$ the \textbf{pivot} and $a,c$ the \textbf{switched pair}.
																											\end{definition}
																											
																											\begin{example}
																												Note that $2'12'\to 12'2'$ is a valid elementary shifted Knuth move of the second type above, since under the standardization ordering the right-hand $2'$ is less than the other $2'$.  On the other hand $212\to 122$ is not valid, because the right-hand $2$ is considered greater than the other $2$.
																											\end{example}
																											
																											Sagan and Worley (\cite{Sagan}, \cite{Worley}), in slightly different notation, showed that shifted Knuth equivalence classes on words are in one-to-one correspondence with semistandard shifted straight shape tableaux, via either JDT rectification or mixed insertion. We will show that the walk's endpoint is an invariant of the equivalence class.
																											
																											\begin{proposition}\label{prop:walks}
																												The endpoint of the lattice walk of a word $w$ is invariant under all shifted Knuth moves.
																											\end{proposition}
																											It will then suffice to check Theorem \ref{thm:rectification} on words of rectified tableaux. 
																											
																											\begin{proof}
																												The first two steps of the walk always lie on the axes, so those steps follow the rules for $xy = 0$, regardless of which letters occur. Thus from now on we assume the switched pair are not the first two letters of the word; in particular, upon reaching the pair, the walk's location is not the origin.
																												
																												Let $a,c$ be the switched pair; let $(x,y)$ be the location of the walk just before the switched pair, and let $b$ be the pivot. We show that the endpoint of the triple is preserved. Note that, for any such triple, it suffices to show that the directions of the arrows of the switched pair are unchanged, and in most cases below (but not all) this will hold.
																												
																												Note that applying $\eta$ has the effect of reversing the standardization numbering (and transposing the walk and its endpoint). In particular, $\eta$ takes Knuth moves to Knuth moves. We use this symmetry to reduce the number of cases to consider.
																												
																												\textbf{Case 1:} $\{a,c\} = \{1',2\}$. There is nothing to check because the directions of $1'$ and $2$ arrows never change.
																												
																												\textbf{Case 2:} $\{a,c\} = \{2',2\}$ or $\{1',1\}$. By applying $\eta$, it suffices to consider $\{2',2\}$, say $a=2'$ and $c=2$, and to show that the $a=2'$ arrow does not change directions after the switch. The standardization order forces the pivot to be on the left (i.e. the exchange is $bca \leftrightarrow bac$) and to be $b = 2'$ or $2$; in particular $y>0$. If $x=0$, then $a=2'$ contributes $\north{}$ in both cases; if instead $x>0$, then in both cases $a=2'$ occurs off the axes, hence gives $\west{}$.
																												
																												\textbf{Case 3:} $\{a,c\} = \{1,2\}$ or $\{1',2'\}$. By applying $\eta$, it suffices to consider $\{1,2\}$. If the pivot is on the left, then it must be $2'$ or $2$, so again $y>0$ and the reasoning is exactly as in Case 2 but for $a=1$ rather than $a=2'$.  If the pivot is on the right, then it is $b=1$ or $2'$, and we show that the location of the walk just before $b$ is unchanged under $acb\leftrightarrow cab$.   If $y>0$, the arrow $a=1$ is $\south{}$ in both cases. If instead $y=0$ (and so $x \ne 0$ since the location is not the origin), the individual steps change; see Figure \ref{fig:knuthmove-switch12} for the remaining possibilities. 
																												
																												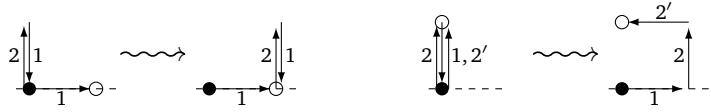
\begin{figure}[t]
																													\begin{center}
																														% up,down,right <> right,up,down (!)
																														\begin{picture}(1,1)(0,0.4)
																														\multiput(-0.2,0)(0.2,0){8}{\line(1,0){0.1}}
																														\put(0,0){\circle*{0.2}}
																														\put(1,0){\circle{0.2}}
																														\stepnorthshift{2}{%
																															\stepsouth{1}{%
																																\stepeast{1}}}
																														\end{picture}
																														$\longsquiggly$
																														\begin{picture}(1,1)(0,0.4)
																														\multiput(-0.2,0)(0.2,0){8}{\line(1,0){0.1}}
																														\put(0,0){\circle*{0.2}}
																														\put(1,0){\circle{0.2}}
																														\stepeast{1}{%
																															\stepnorth{2}{%
																																\stepsouthshift{1}}}
																														\end{picture}
																														\hspace{2cm}
																														% up,down,up <> right,up,left (!!)
																														\begin{picture}(1,1)(0,0.4)
																														\multiput(-0.2,0)(0.2,0){6}{\line(1,0){0.1}}
																														\put(0,0){\circle*{0.2}}
																														\put(0,1){\circle{0.2}}
																														\stepnorthshift{2}{%
																															\stepsouth{\ 1,2'}{%
																																\put(0.1,0){\stepnorth{}}}}
																														\end{picture}
																														$\longsquiggly$
																														\begin{picture}(1,1)(0,0.4)
																														\multiput(-0.2,0)(0.2,0){8}{\line(1,0){0.1}}
																														\put(0,0){\circle*{0.2}}
																														\put(0,1){\circle{0.2}}
																														\stepeast{1}{%
																															\stepnorth{2}{%
																																\stepwest{2'}}}
																														\end{picture}
																													\end{center}
																													\caption{The unique Knuth moves where switching a $1,2$ pair alters the individual steps of the walk. In both cases, the walk begins with $x>0, y=0$. {\bf Left}: The move $211 \leftrightarrow 121$. {\bf Right}: The move
																														$212' \leftrightarrow 122'$.}
																													\label{fig:knuthmove-switch12}
																												\end{figure}
																												
																												\textbf{Case 4:} $\{a,c\} = \{1,2'\}$. The standardization order forces the pivot to be on the right, and it must be $b = 1$ or $2'$. By applying $\eta$, we may assume $b=1$. If $x>1$ and $y>1$, both $a,c$ steps occur off the axes, so we're done; otherwise the path changes in five different ways depending on whether $x,y$ are variously $0$ or $1$ (see Figure \ref{fig:knuthmove-switch12'}). \end{proof}
																											
																											\begin{figure}
																												\begin{center}
																													$x=0, y\geq 1$: \hspace{0.1cm} 
																													% right,left,right <> up,right,down (!)
																													\begin{picture}(1,1)(0,0.4)
																													\multiput(0,-0.2)(0,0.2){7}{\line(0,1){0.1}}
																													\put(0,0){\circle*{0.2}}
																													\put(1,0){\circle{0.2}}
																													\put(0.27,0.33){\scriptsize{1,2',1}}
																													\stepeast{}{%
																														\put(0,0.1){\stepwest{}{%
																																\put(0,0.1){\stepeast{}}}}}
																													\end{picture}
																													\hspace{0.3cm}$\longsquiggly$\hspace{0.3cm}
																													\begin{picture}(1,1)(0,0.4)
																													\multiput(-0.03,-0.2)(0,0.2){7}{\line(0,1){0.1}}
																													\put(0,0){\circle*{0.2}}
																													\put(1,0){\circle{0.2}}
																													\put(-0.3,0.4){\scriptsize{2'}}
																													\stepnorth{}{%
																														\stepeast{1}{%
																															\stepsouth{1}}}
																													\end{picture}
																													\hspace{0.5cm} $x\geq 1, y=0$: \hspace{0.1cm}
																													% right,up,down <> up,down,right (!)
																													\begin{picture}(1,1)(0,0.4)
																													\multiput(-0.2,0)(0.2,0){8}{\line(1,0){0.1}}
																													\put(0,0){\circle*{0.2}}
																													\put(1,0){\circle{0.2}}
																													\stepeast{1}{%
																														\stepnorth{2'}{%
																															\stepsouthshift{1}}}
																													\end{picture}
																													\hspace{0.5cm}$\longsquiggly$\hspace{0.5cm}
																													\begin{picture}(1,1)(0,0.4)
																													\multiput(-0.2,0)(0.2,0){8}{\line(1,0){0.1}}
																													\put(0,0){\circle*{0.2}}
																													\put(1,0){\circle{0.2}}
																													\stepnorth{2'}{%
																														\stepsouthshift{1}{%
																															\stepeast{1}}}
																													\end{picture}
																													\vspace{1cm} \\ %%%
																													$x=1,y>1$: \hspace{0.2cm}
																													% down,left,right <> left,right,down (!)
																													\begin{picture}(1,1)(-1,-0.6)
																													\multiput(-1,-1.2)(0,0.2){7}{\line(0,1){0.1}}
																													\put(0,0){\circle*{0.2}}
																													\put(0,-1){\circle{0.2}}
																													\stepsouth{1}{%
																														\stepwest{2'}{%
																															\stepeastshift{1}}}
																													\end{picture}
																													\hspace{0.3cm}$\longsquiggly$\hspace{0.3cm}
																													\begin{picture}(1,1)(-1,-0.6)
																													\multiput(-1,-1.2)(0,0.2){7}{\line(0,1){0.1}}
																													\put(0,0){\circle*{0.2}}
																													\put(0,-1){\circle{0.2}}
																													\stepwest{2'}{%
																														\stepeastshift{1}{%
																															\stepsouth{1}}}
																													\end{picture}
																													\hspace{0.8cm}
																													$x>1,y=1$: \hspace{0.2cm} 
																													% down,up,down <> left,down,right
																													\begin{picture}(1,1)(0,-0.6)
																													\multiput(-0.3,-1)(0.2,0){7}{\line(1,0){0.1}}
																													\put(0,0){\circle*{0.2}}
																													\put(0,-1){\circle{0.2}}
																													\put(-0.1,0){\stepsouth{}{%
																															\put(0.1,0){\stepnorth{}{%
																																	\put(0.1,0){\stepsouth{1,2',1}}}}}}
																													\end{picture}
																													\hspace{0.1cm}$\longsquiggly$\hspace{1cm}
																													\begin{picture}(1,1)(0,-0.6)
																													\multiput(-1.2,-1)(0.2,0){8}{\line(1,0){0.1}}
																													\put(0,0){\circle*{0.2}}
																													\put(0,-1){\circle{0.2}}
																													\stepwest{2'}{%
																														\stepsouth{1}{%
																															\stepeast{1}}}
																													\end{picture}
																													\vspace{1cm} \\
																													\hspace{0.5cm} $(x,y) = (1,1)$:\hspace{1.2cm} 
																													% down,up,down <> left,right,down (!)
																													\begin{picture}(1,1)(0,-0.6)
																													\multiput(-1,-1.2)(0,0.2){7}{\line(0,1){0.1}}
																													\multiput(-1.2,-1)(0.2,0){8}{\line(1,0){0.1}}
																													\put(0,0){\circle*{0.2}}
																													\put(0,-1){\circle{0.2}}
																													\put(-0.1,0){\stepsouth{}{%
																															\put(0.1,0){\stepnorth{}{%
																																	\put(0.1,0){\stepsouth{1,2',1}}}}}}
																													\end{picture}
																													\hspace{0.1cm}$\longsquiggly$\hspace{1cm}
																													\begin{picture}(1,1)(0,-0.6)
																													\multiput(-1,-1.2)(0,0.2){7}{\line(0,1){0.1}}
																													\multiput(-1.2,-1)(0.2,0){8}{\line(1,0){0.1}}
																													\put(0,0){\circle*{0.2}}
																													\put(0,-1){\circle{0.2}}
																													\stepwest{2'}{%
																														\stepeastshift{1}{%
																															\stepsouth{1}}}
																													\end{picture}
																												\end{center}
																												\caption{Effects on the walk of the Knuth move $12'1 \leftrightarrow 2'11$. There are five cases in which the individual steps change; the starting point $\circ$ must have at least one coordinate equal to $1$ or $0$ (but is assumed not be $(0,0)$).}
																												\label{fig:knuthmove-switch12'}
																											\end{figure}
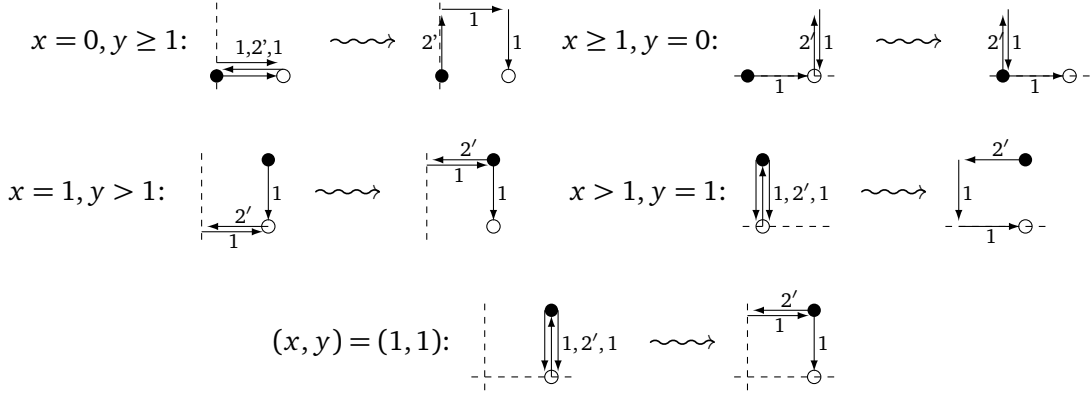
																											
																											\begin{proof}[Proof of Theorem \ref{thm:rectification}]
																												By Proposition \ref{prop:walks}, it suffices to check the statement for words of straight shifted tableaux. In particular, $w$ has the form $2^a 1^a w'$, where $w' = 1^b (2')2^*$, and the $2'$ is optional unless $b=0$, in which case it is required. Let $c$ be the number of $2'/2$s in this suffix, so that $n = 2a + b + c$ and the shape of the tableau is $(a+b+c,a)$. Unwinding the desired statements, we wish to show that the endpoint is $(b,c)$.
																												
																												If $a=0$, the walk simply moves $b$ steps right and $c$ steps up. If $a>0$, the $2^a1^a$ prefix consists of $a$ steps $\north{}$, one step $\east{}$ and $a-1$ steps $\south{}$, ending at $(1,1)$. Then, if $b>0$, the walk moves down and along the $x$-axis to $(b,0)$, then up for the remaining $2'$ and/or $2$ steps to $(b,c)$. If instead $b=0$, the suffix is just $w' = 2'2^{c-1}$, so the walk moves left and up the $y$-axis, ending at $(0,c)$. In all cases the endpoint is $(b,c)$ as desired.
																											\end{proof}
																											
																											Finally, we compute the effects of  $E',F'$ on the endpoint of the lattice walk.
																											
																											\begin{proposition} \label{prop:primed-lattice-endpoint}
																												Applying $F'$ changes precisely one step of the lattice walk, either from $\east{}$ to $\north{}$ or from $\south{}$ to $\west{}$. Consequently, it shifts the endpoint by $(-1,+1)$. For $E'$, the action is reversed.
																											\end{proposition}
																											
																											\begin{proof}
																												It suffices to prove the statement about $F'$. Suppose the $i$-th letter changes. Clearly, the first $i-1$ steps are unchanged, and $w_i$ changes as stated (depending on whether it is on an axis). Finally, from the explicit description of $F'$ in Proposition \ref{prop:explicit-definitions-primed}, every letter after $w_i$ is either $1'$ or $2$, so the step associated to it never changes.
																											\end{proof}
																											
																											Finally, we show that, despite the special rules governing the axes, the shape of a walk does not change drastically if we shift its start point:
																											
																											\begin{proposition}[Bounded error] \label{lem:bounded-error}
																												Let $w$ be a word; consider the walk for $w$ beginning at an arbitrary starting point $(x_0,y_0)$, not necessarily the origin. 
																												
																												If we shift the start by either $\west{}$ or $\north{}$, then the endpoint also shifts by \emph{either} $\west{}$ or $\north{}$.
																												
																												Similarly, if we shift the start by $\east{}$ or $\south{}$, the endpoint also shifts by $\emph{either}$ $\east{}$ or $\south{}$.
																											\end{proposition}
																											
																											\begin{proof}
																												We prove the first statement; the second follows by applying $\eta$. Working inductively, it suffices to show that, for a single letter $a \in \{1',1,2',2\}$, if we shift its starting point by either $\west{}$ or $\north{}$, then its endpoint also shifts by either $\west{}$ or $\north{}$. This is clear if $a = 1'$ or $2$, or if the shift does not move the starting point on or off the axes.
																												
																												We check the other cases. If we shift the starting point $\north{}$ off of the $x$-axis (and we are not on the $y$-axis), we see for both $a = 1,2'$ that the endpoint shifts overall by $\west{}$: 
																												\begin{center}
																													\begin{picture}(0,1)(0,0)
																													\put(0,1){\circle{0.2}}
																													\multiput(-0.3,0)(0.2,0){4}{\line(1,0){0.1}}
																													\put(0.45,0.4){$\leadsto$}
																													\stepnorth{2'}{}
																													\end{picture}
																													\hspace{2cm}
																													\begin{picture}(0,0.5)(0,-1)
																													\multiput(-0.3,-1)(0.2,0){4}{\line(1,0){0.1}}
																													\put(-1,0){\circle{0.2}}
																													\stepwest{2'}{}
																													\end{picture}
																													\hspace{2cm}
																													\begin{picture}(0,1)(0,0)
																													\multiput(-0.3,0)(0.2,0){4}{\line(1,0){0.1}}
																													\put(1,0){\circle{0.2}}
																													\put(1.3,0.4){$\leadsto$}
																													\stepeast{1}{}
																													\end{picture}
																													\hspace{2cm}
																													\begin{picture}(0,0.5)(0,-1)
																													\multiput(-0.3,-1)(0.2,0){4}{\line(1,0){0.1}}
																													\put(0,-1){\circle{0.2}}
																													\stepsouth{1}{}
																													\end{picture}
																												\end{center}
																												Similarly, if $a$ moves $\west{}$ onto the $y$-axis (and was not already on the $x$-axis), in both cases $a = 1,2'$ the endpoint shifts (overall) by $\north{}$.
																											\end{proof}
																											
																											\begin{corollary} 
																												\label{cor:concat-ballot}
																												A concatenation of ballot words is ballot.
																											\end{corollary}
																											\begin{proof}
																												Suppose $w_1, w_2$ are ballot, and let $(x_1,0)$, $(x_2,0)$ be the endpoints of their lattice walks. Then, in the lattice walk for $w_1w_2$, the $w_2$ portion begins at $(x_1,0)$ rather than the origin. By Lemma \ref{lem:bounded-error}, its new endpoint moves $x_1$ steps down and/or right, but clearly can only move right. Thus the endpoint of $w_1w_2$ is at $(x_1+x_2,0)$.
																											\end{proof}

																											%%%%%%%%%%%%%%%%
																											\section{The operators $E$ and $F$}\label{sec:unprimed-ops}
																											
																											As in the previous section, throughout this section we assume that $w$ is a word consisting only of letters from $\{1',1,2',2\}$, and for simplicity we use the following notation.
																											
																											\begin{definition}
																												Define $E=E_1$, $F=F_1$ and $\eta = \eta_1$.
																											\end{definition}
																											
																											There are several parts to this section: we first define $E$ and $F$ and show that they are partial inverses, then prove the main theorem (that $E$ and $F$ are coplactic) in several steps. We show that the operators preserve semistandardness; we then show that they preserve the dual equivalence class of $T$, and finally that they are coplactic.
																											
																											%%%%%%%%%%%%%%%%
																											\subsection{Critical substrings and definition of $E$ and $F$}\label{sec:EFdefinition}
																											
																											\begin{definition}
																												A \defn{substring} of a word $w$ is a consecutive string of letters $u = w_k w_{k+1} \dots w_l$ of some representative of the word $w$.   We say $(x,y) = P_{k-1}(w)$ is the \defn{location} of $u$.
																											\end{definition}
																											
																											\begin{definition} We say that $u$ is an \defn{$F$-critical substring} if certain conditions on $u$ and its location are met. There are five types of $F$-critical substrings. Each row of the first table in Figure \ref{fig:criticals} describes one type of $F$-critical substring and a transformation that can be performed on that type of $F$-critical substring.
																											\end{definition} 
																											
																											\begin{figure}
																												\begin{center}
																													\begin{tabular}{|c|c|c|c|c|}
																														\hline
																														\multirow{2}{*}{Type} 
																														& \multicolumn{3}{c|}{Conditions} & \multirow{2}{*}{Transformation}  \\
																														& \multicolumn{1}{c}{Substring}&  \multicolumn{1}{c}{Steps} & Location &  
																														\\\hline
																														\hline
																														\multirow{2}{*}{1F} & 
																														\multirow{2}{*}{$u = 1(1')^*2'$} &
																														\east{1} ~ \east{1'} ~ \north{2'} & 
																														$y=0$ & 
																														\multirow{2}{*}{$u \to 2'(1')^*2$} \\[.5ex]\cline{3-4}
																														& &  
																														\south{1} ~ \east{1'} ~ \north{2'} & 
																														$y=1$, $x \geq 1$ & 
																														\\[.5ex]\hline
																														\multirow{2}{*}{2F} &
																														\multirow{2}{*}{$u = 1(2)^*1'$} &
																														\east{1} ~ \north{2} ~ \east{1'} & 
																														$x = 0$ & 
																														\multirow{2}{*}{$u \to 2'(2)^*1$}  \\[.5ex]\cline{3-4}
																														&& 
																														\south{1} ~ \north{2} ~ \east{1'} &
																														$x = 1$, $y \geq 1$ & 
																														\\[.5ex]\hline
																														3F & $u = 1$ & 
																														\east{1} & 
																														$y = 0$ & 
																														$u \to 2$ 
																														\\\hline
																														4F & 
																														$u = 1'$ & 
																														\east{1'} & 
																														$x  = 0$ & 
																														$u \to 2'$ 
																														\\\hline
																														\multirow{2}{*}{5F} & $u = 1$ 
																														& \south{1}
																														& \multirow{2}{*}{$x=1$, $y \geq 1$} 
																														&
																														\multirow{2}{*}{undefined} \\[.5ex]\cline{2-3}
																														& $u=2'$ & \west{2'} &&
																														\\\hline
																													\end{tabular}
																												\end{center}
																												\vspace{0.5cm}
																												
																												\begin{center}
																													\begin{tabular}{|c|c|c|c|c|}
																														\hline
																														\multirow{2}{*}{Type} 
																														& \multicolumn{3}{c|}{Conditions} & \multirow{2}{*}{Transformation}  \\
																														& \multicolumn{1}{c}{Substring}&  \multicolumn{1}{c}{Steps} & Location &  
																														\\\hline
																														\hline
																														\multirow{2}{*}{1E} & 
																														\multirow{2}{*}{$u = 2'(2)^*1$} &
																														\north{2'} ~ \north{2} ~ \east{1} & 
																														$x=0$ & 
																														\multirow{2}{*}{$u \to 1(2)^*1'$} \\[.5ex]\cline{3-4}
																														& &  
																														\west{2'} ~ \north{2} ~ \east{1} & 
																														$x=1$, $y \geq 1$ & 
																														\\[.5ex]\hline
																														\multirow{2}{*}{2E} &
																														\multirow{2}{*}{$u = 2'(1')^*2$} &
																														\north{2'} ~ \east{1'} ~ \north{2} & 
																														$y = 0$ & 
																														\multirow{2}{*}{$u \to 1(1')^*2'$}  \\[.5ex]\cline{3-4}
																														&& 
																														\west{2'} ~ \east{1'} ~ \north{2} &
																														$y = 1$, $x \geq 1$ & 
																														\\[.5ex]\hline
																														3E & $u = 2'$ & 
																														\north{2'} & 
																														$x = 0$ & 
																														$u \to 1'$ 
																														\\[.5ex]\hline
																														4E & 
																														$u = 2$ & 
																														\north{2} & 
																														$y  = 0$ & 
																														$u \to 1$ 
																														\\[.5ex]\hline
																														\multirow{2}{*}{5E} & $u = 1$ 
																														& \south{1}
																														& \multirow{2}{*}{$y=1$, $x \geq 1$} 
																														&
																														\multirow{2}{*}{undefined} \\[.5ex]\cline{2-3}
																														& $u=2'$ & \west{2'} &&
																														\\\hline
																													\end{tabular}
																												\end{center}
																												
																												\caption{\label{fig:criticals} Above, the table of $F$-critical substrings and their transformations.  Below, the table of $E$-critical substrings and their transformations.  Here $a(b)^*c$ means any string of the form $abb \dots bc$, including $ac$, $abc$, $abbc$, etc.}
																											\end{figure}

																											Note that all conditions on the starting location are on the lines $y=0$, $y=1$, $x=0$, or $x=1$.  Critical substrings can only occur
																											when the walk touches one of these four lines, which corresponds to the initial subword $w_1w_2 \dots w_{k-1}$ being either ballot, opposite ballot, or close to one of these.  (See Figure \ref{fig:criticals}.)
																											
																											The \textbf{final} $F$-critical substring of a word $w$ is the one with the latest starting index, taking the longest one in the case of a tie. We can now define the operation $F$.
																											
																											\begin{definition}\label{def:F}
																												Let $u$ be the final $F$-critical substring of $w$. Apply the transformation to $u$ according to its type. The resulting word is $F(w)$.  If the type is 5F, or if $w$ has no $F$-critical substrings, then this is undefined and $F(w) = \varnothing$.
																											\end{definition}
																											
																											\begin{remark}
																												In one case, there is a tie for the final critical substring $u$: if $u$ is the first letter of $w$. In this case $u$ counts as both $1$ (3F) and $1'$ (4F). In this case, both outputs represent the same word $F(w)$.
																											\end{remark}
																											
																											\begin{remark}
																												By definition, the critical substring is chosen from among all representatives of $w$. For example, the final $F$-critical substring of $w = 121$ is $12'$ from the representative $12'1$, giving $F(w) = 2'21$ (or $221$ in canonical form). We cannot consider the representatives separately: for example, the last critical substring of the representative $1'2'1$ is $1'$ (type 4F), but it transforms (incorrectly) to the word $2'2'1$, which is not equivalent to $221$.
																											\end{remark}
																											
																											\begin{remark}
																												The proofs in this section make heavy use of every part of the definition of $F$: the substrings, locations and transformation rules, and the use of the \emph{final} critical string. The reader may find it helpful to make a copy of Figure \ref{fig:criticals} for reference.
																											\end{remark}
																											
																											\begin{example} \label{exa:apply-F}
																												Let $w=1221'1'111'1'2'2222'2'11'1$ be the word in Example~\ref{ex:walk}.  There are two $F$-critical substrings:
																												\begin{itemize}
																													\item 
																													the substring $w_1 = 1$ is critical of type 3F (and 4F in a different representative),
																													\item the substring $w_1w_2 = 12'$ (in a different representative) is type 1F,
																													\item the substring $w_1 \cdots w_4 = 1221'$ is type 2F,
																													\item 
																													the substring $w_7 \cdots w_{10} = 11'1'2'$ is type 1F.
																												\end{itemize}
																												The latter is final, so to obtain $F(w)$ we use transformation 1F and change $11'1'2'$ to $2'1'1'2$. We may continue to apply $F$ until it is undefined, reaching the end of the $F$-chain. The result is shown in Figure \ref{fig:F-on-walks}. 
																												
																												\begin{figure}[t]
																													\begin{center}
																														\setlength{\unitlength}{2.5em}
																														\begin{picture}(4,5)(0,0)
																														\multiput(0,0)(0,0.2){32}{\line(0,1){0.1}}
																														\multiput(0,0)(0.2,0){22}{\line(1,0){0.1}}
																														\put(0,0){\circle*{0.13}}
																														\put(3,3){\circle{0.13}}
																														\stepeast{1}{%
																															\stepnorth{2}{%
																																\stepnorth{2}{%
																																	\stepeast{1'}{%
																																		\stepeast{1'}{%
																																			\stepsouth{1}{%
																																				\stepwestshift{2'}{%
																																					\stepeast{1'}{%
																																						\stepeast{1'}{%
																																							\stepnorth{2}{%
																																								\stepnorth{2}{%
																																									\stepnorth{2}{%
																																										\stepnorth{2}{%
																																											\stepwest{2'}{%
																																												\stepwest{2'}{%
																																													\stepsouth{1}{%
																																														\stepeast{1'}{%
																																															\stepsouth{1}{%
																																															}}}}}}}}}}}}}}}}}}
																																															\end{picture} \hspace{1.5cm}
																																															\begin{picture}(3,6)(0,0)
																																															\multiput(0,0)(0,0.2){32}{\line(0,1){0.1}}
																																															\multiput(0,0)(0.2,0){19}{\line(1,0){0.1}}
																																															\put(0,0){\circle*{0.13}}
																																															\put(2,4){\circle{0.13}}
																																															\stepnorth{2}{%
																																																\stepnorth{2}{%
																																																	\stepnorth{2}{%
																																																		\stepeast{1'}{%
																																																			\stepeast{1'}{%
																																																				\stepsouth{1}{%
																																																					\stepwestshift{2'}{%
																																																						\stepeast{1'}{%
																																																							\stepeast{1'}{%
																																																								\stepnorth{2}{%
																																																									\stepnorth{2}{%
																																																										\stepnorth{2}{%
																																																											\stepnorth{2}{%
																																																												\stepwest{2'}{%
																																																													\stepwest{2'}{%
																																																														\stepsouth{1}{%
																																																															\stepeast{1'}{%
																																																																\stepsouth{1}{%
																																																																}}}}}}}}}}}}}}}}}}
																																																																\end{picture}\hspace{2cm}
																																																																\begin{picture}(3.5,6)(0,0)
																																																																\multiput(0,0)(0,0.2){32}{\line(0,1){0.1}}
																																																																\multiput(0,0)(0.2,0){19}{\line(1,0){0.1}}
																																																																\put(0,0){\circle*{0.13}}
																																																																\put(1,5){\circle{0.13}}
																																																																\stepnorth{2}{%
																																																																	\stepnorth{2}{%
																																																																		\stepnorth{2}{%
																																																																			\stepeast{1'}{%
																																																																				\stepeast{1'}{%
																																																																					\stepsouth{1}{%
																																																																						\stepwestshift{2'}{%
																																																																							\stepeast{1'}{%
																																																																								\stepeast{1'}{%
																																																																									\stepnorth{2}{%
																																																																										\stepnorth{2}{%
																																																																											\stepnorth{2}{%
																																																																												\stepnorth{2}{%
																																																																													\stepwest{2'}{%
																																																																														\stepwest{2'}{%
																																																																															\stepwest{2'}{%
																																																																																\stepeastshift{1}{%
																																																																																	\stepsouth{1}{%
																																																																																	}}}}}}}}}}}}}}}}}}
																																																																																	\end{picture}
																																																																																\end{center}
																																																																																\caption{\label{fig:F-on-walks} Successive applications of the operator $F$ to the word $w$ of Example \ref{exa:apply-F}. From left to right, $F(w), F(F(w)), F(F(F(w)))$; we have $F^{(4)}(w) = \varnothing$. Note that each application of $F$ shifts the endpoint by $(-1,+1)$.}
																																																																															\end{figure}
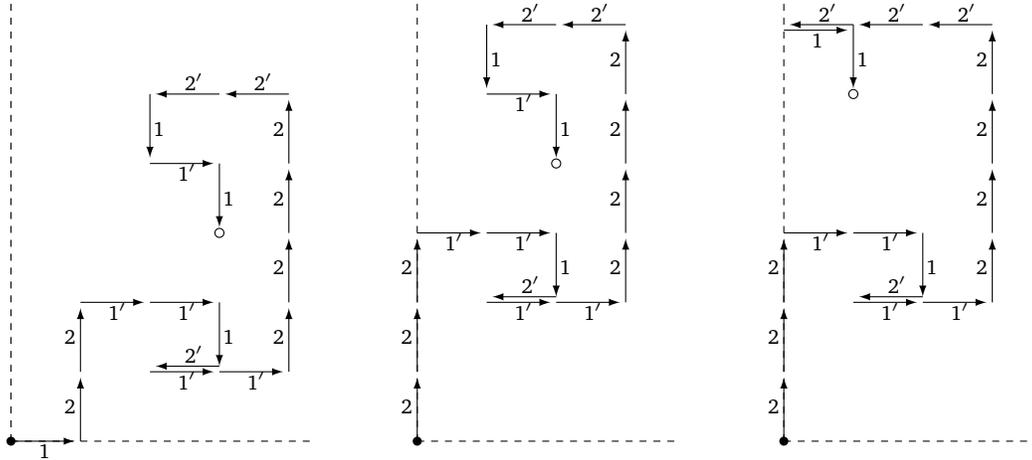
																																																																															
																																																																														\end{example}
																																																																														
																																																																														We now define $E = \eta \circ F \circ \eta$.  There is a corresponding notion of $E$-critical substrings, and transformation rules for each, given by the second table in Figure \ref{fig:criticals}.  This table is obtained by applying $\eta$ to every rule in the table
																																																																														for $F$ (swapping $1' \leftrightarrow 2$, $1 \leftrightarrow 2'$, and $x \leftrightarrow y$).
																																																																														
																																																																														Thanks to the symmetry operator $\eta$, it often suffices to prove statements for $F$ in order to obtain statements for $E$. For ease of exposition, we always opt (when possible) to prove the statements for $F$, not $E$.
																																																																														
																																																																														\subsection{$E$ and $F$ are partial inverses}
																																																																														
																																																																														We now show that $E$ and $F$ are partial inverses:
																																																																														\begin{proposition}\label{prop:inverses}
																																																																															For words $w,w'$, we have $w' = F(w)$ if and only if $E(w') = w$. Moreover, when this holds:
																																																																															\begin{itemize}
																																																																																\item[(i)] The final $E$-critical substring of $F(w)$ is the transformation of the final $F$-critical substring of $w$, and vice versa.
																																																																																\item[(ii)] The operations $E$ and $F$ invert one another case-by-case: $1E \leftrightarrow 2F$, $2E \leftrightarrow 1F$, $3E \leftrightarrow 4F$, $4E \leftrightarrow 3F$.
																																																																															\end{itemize}
																																																																														\end{proposition}
																																																																														
																																																																														It is clear that the case-by-case transformation rules invert one another; the main challenge is to show that no new critical strings are created later in the word. It will suffice (by applying $\eta$) to show that $F(E(w)) = w$ whenever $E(w) \ne \varnothing$. We first make the following observation about the effects of $E,F$ on the substrings of $w$ not containing critical strings. Let $s$ be a substring $w_k \cdots w_l$ of $w$.  Write $(x_i,y_i)$ for the location of the walk just before the letter $w_i$.

																																																																														\begin{lemma} \label{lem:shifting-walks-F}
																																																																															Suppose $x_k\ge 1$ and $(x_k,y_k)\neq (1,0)$. The following are equivalent:
																																																																															\begin{itemize}
																																																																																\item[(i)] There are no $F$-critical substrings in the portion of $w$ corresponding to $s$ and no type 1F critical substrings whose last letter $2'$ is in $s$;
																																																																																\item[(ii)] The walk of the string $s$, but starting from $(x_k-1,y_k+1)$, has locations $(x_i-1,y_i+1)$ for each $i = k, \ldots, l+1$ (this also implies $x_i > 0$ for these $i$).
																																																																															\end{itemize}
																																																																														\end{lemma}
																																																																														In other words, shifting the starting location of $s$ by $(-1,+1)$ results in shifting that entire portion of the walk by $(-1,+1)$.
																																																																														
																																																																														\begin{proof}
																																																																															(ii) $\Rightarrow$ (i): Suppose for contradiction that $s$ contains an $F$-critical substring or the last letter of a 1F.  It is straightforward to check that shifting by $(-1,+1)$ alters the walk for critical strings of type 1F (in particular the last letter $2'$ always changes direction), second kind of 2F (the first letter changes direction), and 5F, giving a contradiction. For 3F, the same is true unless the location is $(1,0)$; but we assumed $s$ does not start at $(1,0)$ and so there is a letter prior to the $1$ in $s$.  Since $x_k>0$ this letter must be a $1$ in location $(1,1)$, giving a 5F in $s$.  The other two critical strings (first kind of 2F; 4F) always begin in location $x=0$, whereas our walk has $x_i > 0$ throughout.
																																																																															
																																																																															(i) $\Rightarrow$ (ii):  Suppose there are no $F$-critical substrings in $s$.  Then first, note that any $\east{}$ arrows starting at $y=0$ in $s$ are $1'$ and not $1$ (so as to avoid 3F substrings) and if there were a $\north{}$ arrow starting at $y=0$, then it is a $2$.  Indeed, if the $\north{}$ is $w_i=2'$, let $j<i$ be the largest index less than $i$ for which $w_j=1$ (this must exist since the word is in canonical form).  Then since $y_i=0$ we cannot have $2'$ or $2$ between $w_j$ and $w_i$, and so the portion of the walk from $w_j$ to $w_i$ looks like $1(1')^\ast 2'$ starting at $y=0$ or $y=1,x\ge 1$.  This is a type 1F, giving an ending of a 1F in $s$, contradicting (i).
																																																																															
																																																																															Now, (ii) holds for $i=k$ since $x_k>0$.  With this as a base case, we induct on $i$ to show (ii) for $k\le i\le l$.  It suffices to show that either $x_i=1$ and $w_i\in \{1',2\}$, or  $x_i\ge 2$, since if $y_i=0$ then we are done since $w_i=1'$ or $2$ by the above argument, and if $y_i>0$ then $w_i$ either doesn't change or starts off the axes both before and after the shift.
																																																																															
																																																																															Assume these conditions hold for $i-1$.  Then $x_i\ge 1$ since we can't move left from the position $x_{i-1}=1$.  If $x_i=1$ and $w_i\in \{1,2'\}$ then $y\ge 1$ (from analyzing $x_{i-1}$ and $w_{i-1}$) and we have a 5F critical substring, a contradiction.  Thus either $x_i=1$ and $w_i\in \{1',2\}$, or $x_i\ge 2$, as desired.
																																																																														\end{proof}
																																																																														
																																																																														We immediately obtain the counterpart of this lemma by applying $\eta$.  (Recall that the walk of $\eta(w)$ is obtained by reflecting the walk of $w$ over the line $y=x$, and the $E$-critical substrings of $w$ of type 1E through 5E correspond to the $F$-critical substrings of $\eta(w)$ of type 1F through 5F.)
																																																																														
																																																																														\begin{lemma} \label{lem:shifting-walks-E}
																																																																															Suppose $y_k\ge 1$ and $(x_k,y_k)\neq (0,1)$. The following are equivalent:
																																																																															\begin{itemize}
																																																																																\item[(i)] There are no $E$-critical substrings in the portion of $w$ corresponding to $s$ and no type 1E critical substrings whose last letter $1$ is in $s$;
																																																																																\item[(ii)] The walk of $s$, but starting from $(x_k+1,y_k-1)$, has locations $(x_i+1,y_i-1)$ for each $i = k, \ldots, l+1$ (this also implies $y_i > 0$).
																																																																															\end{itemize}
																																																																														\end{lemma}
																																																																														
																																																																														We now use these two lemmas repeatedly to prove that $E$ and $F$ are partial inverses.
																																																																														
																																																																														\begin{proof}[Proof of Proposition \ref{prop:inverses}]
																																																																															It suffices to show (by applying $\eta$) that $E(F(w)) = w$ whenever $F(w)$ is defined.  So, suppose $F(w)$ is defined and let $v=F(w)$ where $u=w_k,\ldots, w_l$ is the final $F$-critical substring of $w$, changed to $v_k,\ldots, v_l$.  It suffices to show $v$ has no later $E$-critical substrings.
																																																																															
																																																																															\textbf{Case 1:}  Suppose $u=w_k,\ldots,w_l=1(1')^\ast 2'$ has type 1F.  Then there are no later $F$-critical substrings (nor endings of 1F critical substrings) among the tail $w_{l+1},\ldots,w_n$, and clearly $w_{l+1}$ does not start at $(1,0)$ or on the $y$-axis, so by Lemma \ref{lem:shifting-walks-F}, the tail can be shifted by $(-1,+1)$. Then, by Lemma \ref{lem:shifting-walks-E}, there are no $E$-critical substrings in $v_{l+1} \cdots v_n$. 
																																																																															
																																																																															It now remains to check that there are no $E$-critical substrings starting in $v_{k+1} \cdots v_l = (1')^*2$.  But an inspection of the definition of $E$ shows that the only possibility is a 4E at $v_l$ (since $v_k=2'$ and therefore $v_l=2$ cannot be represented as $2'$).  But $v_k=2'$ makes $y=0$ impossible before $v_l$, and we're done.
																																																																															
																																																																															\textbf{Case 2:}  Suppose $u=w_k,\ldots,w_l=1(2)^\ast 1'$ is type 2F.  Then $w_l=1'$ starts at a location not at $x=0$ or $(1,0)$ except when the location just before $w_k$ is $(1,1)$ and $l=k+1$, and in any case $w_{l+1}$ starts at a valid location for Lemma \ref{lem:shifting-walks-F}.  
																																																																															
																																																																															So, as in Case 1, applying Lemmas \ref{lem:shifting-walks-F} and \ref{lem:shifting-walks-E} we find that $v_{l+1},\ldots,v_n$ has no $E$-critical substrings, and no longer 1E criticals start among $v_k,\ldots,v_l$.  There are clearly no $E$-critical substrings of type 3E, 4E, or 5E starting after the first letter of $v_k,\ldots,v_{l}=2'(2)^\ast 1$ since the locations of the letters involved are not valid for these types.  
																																																																															
																																																																															Finally, the only way that a later 2E can start among $v_k,\ldots,v_l$ is if $v_k=2'$ starts the substring, $l=k+1$, and $v_l=v_{k+1}=1$ is the first $1$ or $1'$ of the word so that it counts as a $1'$ in the 2E critical substring.  For this to happen the $1$ must start on the $y$-axis and by the 2E location conditions this means $v_k=v_1$ starts at $(0,0)$.  But then the 2E critical substring $v_1,\ldots,v_m=21(1')^\ast 2$ changes to $1(1')^\ast 2$ starting at $(0,0)$ in $w$, which counts as a later 1F, a contradiction. 
																																																																															
																																																																															\textbf{Case 3:}  Suppose $u=w_k=1$ is type 3F.  As long as $(x_k,y_k)\neq (0,0)$, the tail $w_{k+1},\ldots, w_{n}$ has a valid starting location for Lemma \ref{lem:shifting-walks-F}, and so by the same arguments as in the previous two cases (since $u$ cannot be extended to a longer 1F) we see that $v_{k+1},\ldots,v_n$ has no later $E$-critical substrings.  In addition $v_k=2$ (a type 4E critical) is not the start of a longer 2E critical substring since then $v_k=2=2'$ is the first $2$ or $2'$ in $v$, and so the first $2$ after it in the 2E critical substring is the first $2$ or $2'$ in $w$, creating a 1F critical substring in $w$ starting at $w_k$.
																																																																															
																																																																															Otherwise, if $(x_k,y_k)=(0,0)$ (and so $k=1$) we see that the next entry $w_2$ starts on an axis in both $w$ and $v$, and hence does not change direction.  The tail starting from $w_3$ also starts in a valid location for Lemma \ref{lem:shifting-walks-F}, so the tail can be shifted and by Lemma \ref{lem:shifting-walks-E} the corresponding tail in $v$ has no $E$-critical substrings.  Any $E$-critical substring starting at $v_2$ must be a 1E or 3E with $v_2=2'$, but then $w_1w_2=12'$ is a 1F critical substring in $w$, a contradiction.  
																																																																															
																																																																															Finally, suppose there is a longer 1E or 2E critical substring starting at $v_1=2=2'$.  If it is 1E of the form $2'(2)^\ast 1$ with at least one $2$ in $(2)^\ast$, then $w$ starts with $1(2)^\ast 1$ and the last $1$ is a type 5F critical substring, a contradiction.  If it is $2'1$ then $w$ starts $11$ and $w_2$ is a later $F$-critical, also a contradiction.  Finally if it is 2E of the form $2'(1')^\ast 2$ then $w$ starts $1(1')^\ast 2$ and this is a longer 1F critical substring (since the first $2$ counts as $2'$).
																																																																															
																																																																															\textbf{Case 4:}  Suppose $u=w_k=1'$ is type 4F.  Note that the only location for which the tail $w_{k+1},\ldots,w_{n}$ does not satisfy the conditions of Lemma \ref{lem:shifting-walks-F} is when the $1'$ arrow starts from $(0,0)$ and so $k=1$.  But this is identical to the analysis in Case 3 since in this situation $u$ is both a 3F and 4F critical substring.
																																																																															
																																																																															So it suffices to consider the case when the tail of $w$ starts off-axes.  By Lemmas \ref{lem:shifting-walks-F} and \ref{lem:shifting-walks-E} there are no $E$-criticals among $v_{k+1},\ldots,v_n$.  To check that no longer $E$-criticals start from $v_k=2'$, note that 1E types are eliminated by Lemma \ref{lem:shifting-walks-E} and 2E types are in the wrong starting location.
																																																																														\end{proof}
																																																																														
																																																																														\subsection{Interaction with the endpoint of the lattice walk}
																																																																														
																																																																														The operators $E,F$ have well-behaved interactions with the shape and endpoint of the lattice walk. We state two useful facts, which follow from Lemmas \ref{lem:shifting-walks-F} and \ref{lem:shifting-walks-E}.
																																																																														
																																																																														\begin{corollary}\label{cor:F-on-walk}
																																																																															The operation $F$ changes the direction of precisely one step in the walk, either from $\south{}$ to $\west{~~}$ or from $\east{~~}$ to $\north{}$. In particular, the endpoint changes by $(-1,+1)$. For $E$, the statement is reversed.
																																																																														\end{corollary}
																																																																														
																																																																														\begin{corollary}
																																																																															\label{cor:initial-subwords-dual}
																																																																															The sum of the coordinates of the $k$th position in the walk is unchanged for all $k$ by $F$ and $E$. It follows that the rectification shape of any initial subword of $w$ is an invariant for $F$ and $E$.
																																																																														\end{corollary}
																																																																														
																																																																														We also give an alternate criterion for when $E,F$ are defined, in terms of the lattice walk rather than strings. We will use this in our proof that $E,F$ are coplactic.
																																																																														
																																																																														\begin{proposition}
																																																																															\label{prop:alternate-definedness}
																																																																															Let $w$ be a word and let $(x,y)$ be the endpoint of $w$'s lattice walk.
																																																																															\begin{itemize}
																																																																																\item[(i)] If $x=0$, then $F(w)$ is undefined.
																																																																																\item[(ii)] If $x = 1$ and the walk does not contain a $\south{}$ or $\west{}$ step, then $F(w)$ is defined.
																																																																																\item[(iii)] If $x=1$ and the walk contains a $\south{}$ or $\west{}$ step, then $F(w)$ is defined if and only if $F'(w)$ is undefined.
																																																																																\item [(iv)] If $x \geq 2$, then $F(w)$ is defined.
																																																																															\end{itemize}
																																																																															The statement for $E$ is given by applying $\eta$ (switching $x \leftrightarrow y$ and $F' \leftrightarrow E'$).
																																																																														\end{proposition}
																																																																														
																																																																														\begin{lemma} \label{lem:xgeq2-Fcritical}
																																																																															Suppose some step $w_i$ of $w$ has $x=0$, and $x \geq 2$ at some later point (possibly the endpoint). Then there is an $F$-critical string in the suffix $w_i w_{i+1} \cdots$.
																																																																														\end{lemma}
																																																																														
																																																																														\begin{proof}[Proof of Lemma]
																																																																															First, skipping any immediate $\north{}$ steps, we may assume $w_i = 1'$ or $1$. If $w_i = 1'$, it is type 4F. This includes the case where $w_i$ is the first $1$ of $w$ (since it counts as both), so in the case $w_i = 1$ we may assume $y>0$. Let $w_j$ be the next letter that is not $2$ ($w_j$ exists because the endpoint does not have $x=1$). If $w_j = 1'$, $w_i \cdots w_j$ is 2F. If $w_j = 1$ or $2'$, then $w_j$ is 5F.
																																																																														\end{proof}
																																																																														
																																																																														\begin{proof}[Proof of Proposition \ref{prop:alternate-definedness}]
																																																																															For (i), we have $F(w) = \varnothing$ because the endpoint cannot shift by $(-1,+1)$. For (ii), by inspection, $w$ has a single $1$, which counts as a 4F-critical string and is final. For (iii) and (iv), the first $1$ of $w$ again counts as 4F, so $F$ will be undefined if and only if the final critical string is type 5F. We consider (iii) and (iv) separately.
																																																																															
																																																																															(iii): We show that the final $F$-critical string is a 5F if and only if the last $1$ or $2'$ in $w$ is a $1$ (that is, $F'$ is defined).
																																																																															
																																																																															$(\Rightarrow)$: Suppose the final critical string $w_i$ has type 5F. If $w_i = 1$, let $w_j$ be the next letter that is not $2$. Then $w_j$ cannot be $1$ (type 5F or 3F, depending on whether $y=0$), $1'$ (type 2F), or $2'$ (type 5F or 1F, depending on whether $y=0$). So $w_j$ cannot exist and the entire suffix contains only $(2)$s. If $w_i = 2'$ instead, then $x=0$ after it, so by Lemma \ref{lem:xgeq2-Fcritical} we must have $x \leq 1$ for the remainder of the word. But then $x=0$ immediately after the last $2'$ of $w$, and there must be a $1$ or $1'$ later to reach $x=1$ at the endpoint. But it cannot be $1'$ (type 4F), hence is a $1$ as desired. 
																																																																															
																																																																															($\Leftarrow$): Suppose the last $1$ or $2'$ is $w_i = 1$ at location $(x_i,y_i)$. Since all further steps are $1'$ (right) and $2$ (up), we must have $x_i=0$ or $1$. Either way, $x_{i+1}=1$, so in fact the suffix is just $(2)^*$. If $x_i=1$, then $w_i$ is type 5F and we're done. If $x_i=0$, consider the largest index $j < i$ such that $x_j=1$ ($j$ exists because the walk must be off-axes at some point if it has a $\south{}$ or $\west{}$ step). Then $w_j = 2'$ is 5F-critical and is final by inspection.
																																																																															
																																																																															(iv) Suppose $x \geq 2$ and the final $F$-critical string $w_i$ has type 5F. If $w_i = 2'$, Lemma \ref{lem:xgeq2-Fcritical} gives a later critical string (since $x_{i+1}=0$), so we must have $w_i = 1$. Let $w_j$ be the next letter that is not $2$, which must exist because the endpoint has $x \geq 2$. But we cannot have $w_j = 1'$ (type 2F) or $1$ (type 3F or 5F, depending on whether $y=0$) or $2'$ (type 1F or 5F, depending on whether $y=0$), a contradiction.
																																																																														\end{proof}
																																																																														
																																																																														As a corollary, we take the first step toward showing that $E$ and $F$ are coplactic:
																																																																														
																																																																														\begin{corollary} \label{cor:Fdefinedness-knuth}
																																																																															Let $w, w'$ be Knuth equivalent. Then $F(w)$ is defined iff $F(w')$ is defined.
																																																																														\end{corollary}
																																																																														
																																																																														\begin{proof}
																																																																															By Proposition \ref{prop:alternate-definedness}, whether $F$ is defined depends only on the endpoint of the walk, whether the walk contains any $\south{}$ or $\west{}$ steps, and whether $F'$ is defined. This data is preserved under Knuth moves by Propositions \ref{prop:primed-coplactic} and \ref{prop:walks} and Theorem \ref{thm:rectification}.
																																																																														\end{proof}
																																																																														
																																																																														Finally, we obtain a characterization of ballotness depending only on our operators:
																																																																														
																																																																														\begin{proposition} \label{prop:ballot-iff-killed}
																																																																															A word $w$ in the alphabet $\{1',1,2',2\}$ is ballot if and only if $E(w) = E'(w) = \varnothing$. It is anti-ballot if and only if $F(w) = F'(w) = \varnothing$.
																																																																														\end{proposition}
																																																																														
																																																																														\begin{proof}
																																																																															We prove the anti-statement. If both are undefined, Proposition \ref{prop:alternate-definedness} shows that the endpoint has $x=0$, hence $w$ is anti-ballot by Corollary \ref{cor:ballot-walk-criterion}. Conversely, if $w$ is anti-ballot, the endpoint has $x=0$. Then $F$ and $F'$ must be undefined because either would cause the endpoint to shift by $(-1,+1)$.
																																																																														\end{proof}
																																																																														
																																																																														\subsection{The operators $E$ and $F$ on tableaux}
																																																																														
																																																																														Let $T$ be a skew, shifted semistandard tableaux with entries in $\{1',1,2',2\}$. We define $E(T), F(T)$ by applying $E,F$ to the row reading word of $T$. (We will see later, Proposition \ref{prop:rows-columns}, that using the column word results in the same action.) The main claim of this subsection is that this gives a well-defined action on \emph{semistandard} tableaux:
																																																																														
																																																																														\begin{theorem} \label{thm:FE-semistandard}
																																																																															The tableaux $F(T), E(T)$ are semistandard (when defined).
																																																																														\end{theorem}
																																																																														
																																																																														We first show the following. Let $a,b \in T$ with $b$ just below $a$.
																																																																														
																																																																														\begin{lemma} \label{lem:special-ssyt-cases}
																																																																															If the lattice walk has $y=0$ at $a$, then $b=1$ or $1'$.
																																																																															
																																																																															Next, suppose $y=0$ at $b$, and $a \in \{1,2'\}$, and there is a $1$ somewhere before $a$ in reading order. Then there is a 5E-critical string in $a$'s row.
																																																																														\end{lemma}
																																																																														
																																																																														\begin{proof}
																																																																															Let $k$ be the number of $(2')$s and $(2)$s to the left of $b$. If the leftmost is not on the diagonal, then there are at most $k$ $(1)$s to the left of $a$; otherwise, there may be $(k+1)$ $(1)$s, but the first contributes $\east{}$ to the walk. Now if $b \in \{2',2\}$, then we have $y \geq k+1$ in the walk just after $b$, hence $y \geq k+1$ at the end of $b$'s row. But then we have at most $k$ downwards steps prior to $a$, contradicting the condition $y=0$.
																																																																															
																																																																															For the second statement, let $k$ be the number of $(2)$s to the right of $b$ (by semistandardness, $b \geq 2'$ and every entry to its right is a $2$). Then from $a$ to the end of the upper row, there are at least $k$ $(1)$s (including $a$ itself if $k>0$), followed by an entry $a' \in \{1,2'\}$:
																																																																															\[\begin{ytableau}
																																																																															a \\
																																																																															b & 2 & 2 & 2
																																																																															\end{ytableau} \qquad \leadsto \qquad 
																																																																															\begin{ytableau}
																																																																															1 & 1 & 1 & a' & \cdots \\
																																																																															b & 2 & 2 & 2
																																																																															\end{ytableau}\]
																																																																															We claim that $x \geq 1, y \leq k+1$ at $a$. Hence $a'$ or some entry before it has $y=1$, giving a 5E-critical string. The condition on $y$ is clear because $y=k+1$ at the end of $b$'s row, and any entry to the left of $a$ is $1$ or $1'$. As for $x$: if $x=0$ at $b$, then $b$ is the first letter of the word, so (by our assumption on $a$) there must be a single $1$ to the left of $a$, giving an initial $\east{}$ step. Otherwise $x \geq 1$ at $b$, hence at the end of $b$'s row, hence at $a$.
																																																																														\end{proof}
																																																																														
																																																																														\begin{proof}[Proof of Theorem \ref{thm:FE-semistandard}]
																																																																															We use the representatives for $T$ and $F(T)$ in which the final $F$-critical string appeared, not necessarily the canonical form. In particular, $F(T)$ has the same entries as $T$ except for the first and last letters of the final $F$-critical string. Note that changing representatives never violates semistandardness.
																																																																															
																																																																															Let $a,b$ be adjacent squares in $T$:
																																																																															$\raisebox{0cm}{\small \begin{ytableau}a&b\end{ytableau}}$ \ or \ $\raisebox{.5\height}{\small \begin{ytableau}a\\ b\end{ytableau}}$. 
																																																																															Let $a',b'$ be the corresponding squares in $F(T)$ or $E(T)$. If neither square changes, there is nothing to check. If both squares change, they must be the first and last letters of a 1F or 2F critical string and, by semistandardness of $T$, must take the form 
																																																																															\[\raisebox{-.25\height}{\begin{ytableau}1 & 2'\end{ytableau}}\ \xrightarrow{\ 1F\ }\
																																																																															\raisebox{-.25\height}{\begin{ytableau}2' & 2\end{ytableau}} \qquad \text{or} \qquad 
																																																																															\raisebox{.25\height}{\begin{ytableau}1' \\ 1\end{ytableau}}
																																																																															\ \xrightarrow{\ 2F\ }\
																																																																															\raisebox{.25\height}{\begin{ytableau}1 \\ 2'\end{ytableau}}\]
																																																																															In these cases we see that semistandardness is preserved. Next, suppose $a$ changes and $b$ does not. For $E(T)$, we're done because the new value always has $a' < a$. For $F(T)$, we have $a' > a$, so there are a few cases.
																																																																															
																																																																															{\bf Horizontal case for $F$}. If $a = 1' \leadsto 1$ or $2' \leadsto 2$ (1F and 2F, last letters), $b$ is large enough ($b \geq 1$ or $2$) because $T$ is semistandard. If $a = 1 \leadsto 2'$ (1F and 2F, first letters), we cannot have $b=2'$ by the definition of the substring (and our assumption that $b$ does not change), so $b=2$ as desired.  Finally, if $a = 1 \leadsto 2$ (3F) or $a = 1' \leadsto 2'$ (4F), suppose $b \in \{1,2'\}$. Then $b$ or $ab$ gives a later or longer 1F, 3F or 5F critical string (depending on which coordinates are zero at $a$).
																																																																															
																																																																															{\bf Vertical case for $F$}. 
																																																																															First, if $a = 1 \leadsto 2'$ (1F and 2F, first letter), we have $b \geq 2'$ by semistandardness. Next, if $a = 2' \leadsto 2$ (1F, last letter) or $1 \leadsto 2$ (3F), Lemma \ref{lem:special-ssyt-cases} forces $b \ne 1',1$ (note that $y=0$ at $a$), so in fact $b$ cannot exist. Finally, if $a = 1' \leadsto 1$ (2F, last letter) or $1' \leadsto 2'$ (4F), suppose $b \in \{1',1\}$. Since $x=0$ at $a$, $b$ cannot be the preceding letter (and $a$ is leftmost in its row), so $b$ has an entry $c$ to its right, with an entry $d$ above it. By semistandardness, $d \in \{1,2'\}$, giving a later 5F (note that $y\ne0$ since $a$ is not the first letter of the word). This completes this case.
																																																																															
																																																																															Finally, suppose $b$ changes and $a$ does not. For $F(T)$, there is nothing to check because $b' > b$. In $E(T)$, we have $b' < b$, so there are a few cases.
																																																																															
																																																																															{\bf Horizontal case for $E$}. If $b = 2' \leadsto 1$ (1E and 2E, first letters), then $a \leq 1$ since $T$ is semistandard. If $b = 2 \leadsto 2'$ (2E, last letter), then $a = 1'$ from the definition of the substring. If $b = 1 \leadsto 1'$ (1E, last letter) then the substring instead indicates $a=2$. This violates semistandardness, so in fact $a$ cannot exist. If $b = 2' \leadsto 1'$ (3E), observe that, since $x=0$ at $b$, we can't have $a \in \{1,1'\}$. So, by semistandardness, $a$ again does not exist. Lastly, if $b = 2 \leadsto 1$ (4E), then $a \notin \{2',2\}$ since we must have $y=0$ at $b$. 
																																																																															
																																																																															{\bf Vertical case for $E$}. If $b = 2 \leadsto 2'$ or $1 \leadsto 1'$ (1E and 2E, last letter), then $a$ is small enough since $T$ is semistandard. In all remaining cases, we have $b \in \{2',2\} \leadsto b' \in \{1',1\}$. We know $a \ne 2$ by semistandardness, but we need to show $a = 1'$. So, assume for contradiction that $a \in \{1,2'\}$; we will find various later $E$-critical strings.
																																																																															
																																																																															First, suppose $b = 2 \leadsto 1$ (4E). There can't be a $1$ or $1'$ prior to $a$ in reading order, since Lemma \ref{lem:special-ssyt-cases} would then find a 5E-critical string in the row of $a$ (since $y=0$ at $b$). But if there are no $(1)$s or $(1')$s before $b$ clearly $x=0$ at $b$ as well, so $b$ is the first letter of the word and counts as a final 3E in a different representative of $T$; we deal with it below.
																																																																															
																																																																															In the remaining cases, $b = 2'$. If there is an entry $\tilde{a}$ to $a$'s left, semistandardness forces $\tilde{a} = 1'$ or $1$, and the square below $\tilde{a}$ (to $b$'s left) to be empty. So in fact Lemma \ref{lem:special-ssyt-cases} applies again ($\tilde{a}$ and $b$ are on the diagonal, so $b$ is the first letter of the word and has location $y=0$). So we may assume $a$ is leftmost in its row and $b \cdots a = 2' (2)^* a$.
																																																																															
																																																																															If $b = 2' \leadsto 1'$ (3E), then either $a=1$, making $b \cdots a = 2' (2)^*1$ type 1E-critical, or $a=2'$ and is itself 3E-critical. If $b = 2' \leadsto 1$ (1E, first letter), then $a$ must be the final $1$ of the 1E string. But this contradicts our assumption that $a$ did not change in $F(T)$.
																																																																															
																																																																															Finally, suppose $b = 2' \leadsto 1$ (2E, first letter). If the location has $x \geq 2$ at $b$, then in fact the same counting argument as in Lemma \ref{lem:special-ssyt-cases} 
																																																																															shows that $x \geq 1$ and $y \leq k+1$ at $a$ (where $k$ is the number of $(2)$s to the right of $b$), yielding a 5E-critical string in $a$'s row. Otherwise, $b$'s location is either $(0,0)$ or $(1,1)$. Either way, note that the 2E must be just $2'2$, with a $2$ to the right of $b$, which forces $a=1$ by semistandardness. Thus $b\cdots a = 2'(2)^*1$ is 1E-critical.
																																																																														\end{proof}
																																																																														
																																																																														Having shown that our operators are well-defined on tableaux, our remaining goal is to show that they are coplactic (Theorem \ref{thm:coplactic}). Before proceeding, we end this section with some final details on the interaction of $E,F$ with the tableau structure. We first describe the action of $E$ and $F$ on straight shifted tableaux.
																																																																														
																																																																														\begin{proposition} \label{prop:rect-unprimed-action} Let $T$ be a straight shifted semistandard tableau with entries $\{1',1,2',2\}$.  If $T$ has no $2'$ in the first row, then $F$ changes the rightmost $1$ to a $2$ (or gives $\varnothing$ if this breaks semistandardness), and $E$ changes the leftmost $2$ in the first row to a $1$ (or gives $\varnothing$ if there is no such $2$). 
																																																																															
																																																																															If $T$ has a $2'$ in the first row, then $F$ changes the substring $12'$ to $2'2$ (or gives $\varnothing$ if this breaks semistandardness), and $E$ does the opposite (or gives $\varnothing$ if there is no $2$).
																																																																														\end{proposition}
																																																																														
																																																																														\begin{proof}
																																																																															We prove the claim about $F$, since $E$ is just the inverse. The word of $T$ is either $1^a2^b$, or $2^a1^b2^c$ with $b>a>0$, or $2^a1^b2'2^c$ with $b \geq a > 0$. In the case $1^a2^b$, if $b>0$, the substring $12$ is type 1F, becoming $22$ (after canonicalizing). If $b=0$, the last $1$ is type 3F. In the case $2^a1^b2^c$, the final $1$ is type 5F if $b=a+1$ and type 3F otherwise. Similarly, in the case $2^a1^b2'2^c$, the substring $2'$ is type 5F if $b=a$; otherwise the substring $12'$ is type 1F. (The edge cases $b=a+1$ and $b=a$, leading to 5F critical strings, are also the cases where semistandardness would break.)
																																																																														\end{proof}
																																																																														
																																																																														We also show that our use of row words rather than column words was harmless. We prove this \emph{assuming} $E$ and $F$ are coplactic (and do not use this statement in the proof of that fact). 
																																																																														
																																																																														Recall that $\etaT$ is the operation on tableaux corresponding to the map $\eta$ on words.
																																																																														
																																																																														\begin{proposition} \label{prop:rows-columns}
																																																																															Let $w, v$ be the row and column reading words of a tableau $T$. Then applying $F$ to $w$ changes the entries of $T$ in the same way as applying $F$ to $v$.
																																																																														\end{proposition}
																																																																														\begin{proof}
																																																																															Let $\tilde{F}(T)$ be obtained by applying $F$ to $v$. Note that if $T$ is a shifted straight shape, then $\tilde{F}(T) = F(T)$ by inspection from Proposition \ref{prop:rect-unprimed-action}.  We now claim that $\tilde{F}$ preserves semistandardness and is coplactic; hence $\tilde{F}(T) = F(T)$ for all tableaux $T$. To show this, we recall the action $\etaT$ on tableaux (Definition \ref{def:eta-on-T}), and we observe that $\tilde{F}(T) = \etaT \circ E \circ \etaT(T)$. The two desired properties hold for $\etaT$ by Remark \ref{rmk:row-col-eta} and for $E$ by Theorems \ref{thm:FE-semistandard} and \ref{thm:coplactic}, hence they hold for the composition.
																																																																														\end{proof}

																																																																														\subsection{Compatibility with dual equivalence}
																																																																														
																																																																														In this section, we show $F(T)$ and $E(T)$ are shifted dual equivalent to $T$ (when defined). We first recall the definition of dual equivalence in terms of mixed insertion and shifted RSK.  We follow the conventions used in \cite{Sagan}; for convenience, we extend the usual definition from standard to semistandard shifted tableaux.
																																																																														
																																																																														\begin{definition}
																																																																															For letters $a$ and $b$ in $\{1'<1<2'<2<3'<3<\cdots\}$, we say that $a\prec_{\mathrm{row}} b$ if either $a$ is a primed letter and $a\le b$, or if $a$ is unprimed and $a<b$.  We say that $a\prec_{\mathrm{col}} b$ if either $a$ is primed and $a<b$ or $a$ is unprimed and $a\le b$.   
																																																																														\end{definition}
																																																																														
																																																																														\begin{definition}
																																																																															Given a shifted semistandard tableau $T$ of straight shape $\lambda$ and a letter $a\in \{1',1,2',2,3',3,\ldots\}$, the \textbf{mixed insertion} of $a$ into $T$ is the tableau formed by the following process.
																																																																															\begin{enumerate}
																																																																																\item Insert $a$ into the first row $R_1$, i.e., let $b$ be the leftmost entry of $R_1$ for which $a\prec_{\mathrm{row}} b$.  If there is no such $b$ then place $a$ at the end of row $R_1$ and stop the process. Otherwise, replace $b$ with $a$ and `bump out' $b$. If $b$ was the leftmost entry in the row, proceed to step 2; otherwise, repeat step 1, inserting $b$ into the next row.
																																																																																\item Insert $b$ into the next \textit{column} to its right, bumping out the topmost entry $y$ for which $b\prec_{\mathrm{col}} y$, or placing $b$ at the bottom of the column if $y$ does not exist. Repeat step 2 (inserting $y$ into the following column) until the process terminates.
																																																																															\end{enumerate}
																																																																														\end{definition}
																																																																														
																																																																														\begin{example}
																																																																															The mixed insertion of $1$ into ${\scriptsize \begin{ytableau}1&1&1&2'&2\\\none & 2&2&2\end{ytableau}}$ is computed as follows.
																																																																															
																																																																															In Step 1, we bump out the $2'$ in the top row and insert it into the diagonal square on the second row, bumping out a $2$:
																																																																															\[{\scriptsize \begin{ytableau}1&1&1&2'&2\\\none & 2&2&2\end{ytableau}}\qquad \longsquiggly \qquad {\scriptsize \begin{ytableau}1&1&1&1&2\\\none & 2&2&2\end{ytableau}} \qquad \longsquiggly \qquad {\scriptsize \begin{ytableau}1&1&1&1&2\\\none & 2'&2&2\end{ytableau}}\]
																																																																															We then column-insert the $2$ in the subsequent column, bumping the $2$ in the next column to the right, and so on, until the last $2$ ends up in a new column. Finally, we canonicalize the tableau:
																																																																															\[{\scriptsize \begin{ytableau}1&1&1&1&2\\\none & 2'&2&2\end{ytableau}}
																																																																															\qquad \longsquiggly\qquad
																																																																															{\scriptsize \begin{ytableau}1&1&1&1&2&2\\\none & 2'&2&2\end{ytableau}}
																																																																															\qquad\longsquiggly\qquad
																																																																															{\scriptsize \begin{ytableau}1&1&1&1&2&2\\\none & 2&2&2\end{ytableau}}.\]
																																																																														\end{example}
																																																																														
																																																																														\begin{definition}
																																																																															The mixed insertion of $a$ into $T$ is \textbf{Schensted} if no column bumping occurs, and it is \textbf{non-Schensted} otherwise.
																																																																														\end{definition}
																																																																														
																																																																														\begin{definition}
																																																																															A \textbf{circled tableau} is a tableau in which some entries may be circled.
																																																																														\end{definition}

																																																																														\begin{definition}
																																																																															Let $w = w_1 \cdots w_n$ be a word.  The \textbf{shifted RSK insertion} of $w$ is the pair $(P,Q)$, with $P$ a semistandard \textbf{insertion tableau} and $Q$ a circled, standard \textbf{recording tableau}, constructed recursively as follows. Start with $(P,Q) =(\varnothing,\varnothing)$. For $i=1,\ldots,n$, mixed-insert $w_i$ into $P$. Then add $i$ to $Q$ in the same outer corner that was filled in $P$ at the $i$-th step, and if that step was non-Schensted, circle $i$ in $Q$.
																																																																														\end{definition}
																																																																														
																																																																														\begin{example}
																																																																															The shifted RSK insertion of $w = 22111'2'1$ is
																																																																															\[\Bigg(\
																																																																															\raisebox{1ex}{\begin{ytableau}
																																																																																1 & 1 & 1 & 1 \\
																																																																																\none & 2 & 2 & 2
																																																																																\end{ytableau}}\ ,\
																																																																															\raisebox{1ex}{\begin{ytableau}
																																																																																1 & 2 & \circled{3} & \circled{5}\\
																																																																																\none & 4 & 6 & \circled{7}
																																																																																\end{ytableau}}\ \Bigg).
																																																																															\]
																																																																														\end{example}
																																																																														
																																																																														\begin{remark}
																																																																															For a skew shifted semistandard tableau $T$, the insertion tableau $P$ of its reading word is the same as the jeu de taquin rectification of $T$ (see \cite{Sagan}, \cite{Worley}). Note also that standardizing $T$ results in standardizing $P$ and leaving $Q$ unchanged.
																																																																														\end{remark}
																																																																														
																																																																														In the case where $w$ only has entries from $\{1',1,2',2\}$, we can easily compute which entries will be circled in the $Q$ tableau.   
																																																																														
																																																																														\begin{lemma}\label{lem:circling}
																																																																															Let $w=w_1\cdots w_n$ be a word in $\{1',1,2',2\}^n$ in canonical form and let $Q$ be its RSK recording tableau.  Then the number $i$ is circled in $Q$ if and only if $w_i$ satisfies one of the following conditions.
																																																																															\begin{enumerate}
																																																																																\item\label{condition1} $w_i=1'$,
																																																																																\item\label{condition2} $w_i=1$ or $2'$, and $i>1$, and $w_1, \ldots, w_{i-1} \in \{2',2\}$,
																																																																																\item\label{condition4} $w_i=1$ or $2'$, and $i > j$, where $w_j$ is the first downwards or leftwards step of the lattice walk of $w$, and the most recent $1$ or $2'$ before $w_i$ is a $2'$. (Equivalently, if $i > j$ and $F'(w_1, \ldots, w_{i-1}) = \varnothing$.)
																																																																															\end{enumerate}
																																																																														\end{lemma}
																																																																														
																																																																														\begin{proof}
																																																																															At step $i-1$ of the RSK process, the recording tableau $P$ is a one or two-row tableau of one of the forms illustrated in Figure \ref{fig:two-row}.  Let $R_1$ and $R_2$ be the two (possibly empty) rows of $P$.  Then the only two possibilities for the entry $i$ being circled in $Q$ is if the insertion of $w_i$ into $P$ either bumps out the first entry of $R_1$, or bumps an entry into $R_2$ which in turn bumps out its first entry.
																																																																															
																																																																															First we analyze the cases in which $w_i$ bumps out the first entry of $R_1$.  If $w_i=1'$ it cannot be the first entry of the word (since $w$ is in canonical form) and is always less than the first entry in $\prec_{\mathrm{row}}$ order, hence always bumps it.     This gives condition \ref{condition1}.   If $w_i=1$ or $2'$, it bumps out the first entry of $R_1$ if and only if $i>1$ and there are no $(1)$s already in $R_1$. Equivalently, there should be no $(1)$s or $(1')$s among $w_1, \ldots, w_{i-1}$, since such entries will always end up in $R_1$.  This gives condition \ref{condition2}.    Finally, $w_i=2$ can never bump out the very first entry in $R_1$. 
																																																																															
																																																																															Now, we consider when the insertion of $w_i$ can bump an entry $b$ in $R_1$ to row $R_2$, and $b$ in turn bumps out the first entry of $R_2$ (which must be a $2$ by semistandardness).  This never occurs if $w_i=1'$, because the algorithm switches to column insertion immediately, nor if $w_i=2$, since then $w_i$ is placed at the end of $R_1$. So suppose $w_i=1$ or $2'$ and $w_i$ bumps out an entry $b$ that is not at the start of row $R_1$. Then $w_i$ is circled iff $b=2'$ (by definition of $\prec_{\mathrm{row}}$).
																																																																															Recall also that there can be at most one $2'$ in $R_1$ (in which case it is the entry $b$ and is bumped out), and this occurs if and only if $P$ has two nonempty rows and $F'(P)=\varnothing$.
																																																																															
																																																																															We are left with determining when this situation occurs. First note that $P$ has two rows if and only if the walk has at least one downwards or leftwards step, by Theorem \ref{thm:rectification}.  Let $w_j$ be the first such step.  Once $w_j$ is inserted (it is never circled), the insertion tableau $P$ has a $2'$ in $R_1$ if and only if $w_j=2'$. Let $w_i$ be the next $1$ or $2'$ after $w_j$. Then we see that $i$ is circled if and only if $w_j=2'$, and when step $i$ is complete, $R_1$ contains a $2'$ if and only if $w_i=2'$.  Continuing this argument to the right gives condition \ref{condition4}.
																																																																														\end{proof}
																																																																														
																																																																														The following was shown in \cite{Haiman}.
																																																																														
																																																																														\begin{proposition}[\cite{Haiman}, Lemma 2.11 and Theorem 2.12] \label{prop:Haiman1}
																																																																															Two shifted standard tableaux of the same skew shape are dual equivalent if and only if their row words have the same mixed-insertion recording tableau.
																																																																														\end{proposition}
																																																																														
																																																																														While the above proposition is only stated for standard tableaux, note that our expanded definition of mixed insertion above is compatible with standardization of tableaux, and so by Lemma \ref{lem:unstandardize} it also holds for semistandard tableaux. We now show:
																																																																														
																																																																														\begin{proposition}
																																																																															When defined, $F(w)$ and $E(w)$ are dual equivalent to $w$.
																																																																														\end{proposition}
																																																																														
																																																																														\begin{proof}
																																																																															It suffices to show that, under mixed insertion, $F(w)$ has the same circled recording tableau as $w$. We first show this when the final $F$-critical string has type 1F or 3F. It then follows for 2F and 4F by applying $\eta$ and Proposition \ref{prop:inverses} (see Remark \ref{rmk:eta-words-dual}).
																																																																															
																																																																															By Corollary \ref{cor:initial-subwords-dual}, the insertion tableaux $P_{w}$ and $P_{F(w)}$ have the same shape at every step of the insertion process. Therefore the recording tableaux $Q_{w}$ and $Q_{F(w)}$ are the same if we ignore the circling. We are left with showing that $Q_{w}$ and $Q_{F(w)}$ have the same circled entries.
																																																																															
																																																																															\textbf{Case 1F:} Suppose the final $F$-critical substring $u$ in $w$ is type 1F, so $u=1(1')^\ast 2'$ with starting location either $y=0$ or $y=1$, $x\ge 1$. Then $u$ changes to $s=2'(1')^\ast 2$ in $F(T)$.  Clearly the circlings prior to $u$ in reading order are unchanged, and the $(1')$s are circled indices in both $u$ and $s$. Any later circlings are also unchanged because, in both $u$ and $s$, the last $1$ or $2'$ in these substrings is a $2'$.
																																																																															
																																																																															We next check that $u_1=1$ is a circled index if and only if $s_1=2'$ is. First note that $u_1$ satisfies condition \ref{condition2} of Lemma \ref{lem:circling} if and only if $s_1$ does. For condition \ref{condition4}, note that the index $j$ of the first downwards or leftwards step in the walk is preserved by $F$, by Corollary \ref{cor:F-on-walk}. Thus $i$ is uncircled if $i \leq j$; if instead $i > j$, then $i$ is circled if and only if the previous $1$ or $2'$ is a $2'$.
																																																																															
																																																																															Finally, since the $2'$ in $u$ follows $u_1=1$, its index is uncircled by conditions \ref{condition2} and \ref{condition4}.  The corresponding $2$ in $s$ is also uncircled because $(2)$s are never circled.
																																																																															
																																																																															\textbf{Case 3F:}  Suppose the final $F$-critical substring $u$ of $w$ is type 3F, so $u=1$, becoming $s=2$ in $F(T)$. Let $j$ be the index of the first downwards or leftwards step in the walk. Note that $i \ne j$ because $y=0$ at $u$, and if $j$ does not exist then we're done by Lemma \ref{lem:circling}.
																																																																															
																																																																															First, $u$ itself cannot satisfy condition \ref{condition2} of Lemma \ref{lem:circling} because it starts at $y=0$. For condition \ref{condition4}, if $i < j$, $i$ is uncircled, and changing $u$ to $s$ does not affect the circling of any later letters (because $w_j$ will be a $1$ or $2'$). If $i > j$, the walk must have had $y>0$ at some point prior to $u$.  To reach $y=0$ at $u$, the last $1$ or $2'$ before $u$ must have been a downwards $1$, with no $(2')$s (or $(2)$s) between it and $u$. But then $u=1$ is an uncircled index (as is $s=2$), and changing $u$ to $s$ does not affect the circling of any later letters.
																																																																														\end{proof}
																																																																														
																																																																														\begin{corollary} \label{cor:FE-dual-equiv}
																																																																															When defined, $F(T)$ and $E(T)$ are dual equivalent to $T$.
																																																																														\end{corollary}
																																																																														
																																																																														\begin{proof}
																																																																															We have shown that $T$ and $F(T)$ are semistandard tableaux of the same skew shape, and that their row words have the same mixed-insertion recording tableau.  The result now follows by Proposition \ref{prop:Haiman1}.
																																																																														\end{proof}
																																																																														
																																																																														\subsection{Coplacticity}
																																																																														
																																																																														We now use the machinery we have developed to prove the main result of this paper, that $F$ and $E$ are coplactic.  We also require the following result of Haiman \cite{Haiman}:
																																																																														
																																																																														\begin{proposition}[\cite{Haiman}, Theorem 2.13] \label{prop:Haiman2}
																																																																															A skew shifted tableau is uniquely determined by its shape, dual equivalence class, and jeu de taquin equivalence class (rectification).
																																																																														\end{proposition}

																																																																														\begin{theorem}\label{thm:coplactic}
																																																																															The operations $F$ and $E$ are coplactic.
																																																																														\end{theorem}
																																																																														
																																																																														\begin{proof}
																																																																															Let $T$ be a tableau and let $S$ be obtained from $T$ by a JDT slide. If $F(T) = F(S) = \varnothing$, there is nothing to show, so assume $F(T) \ne \varnothing$. By Corollary \ref{cor:Fdefinedness-knuth}, $F(S) \ne \varnothing$ as well. Let $\mathrm{slide}(F(T))$ be the tableau obtained from $F(T)$ by the slide initiated in the same square. We show $\mathrm{slide}(F(T)) = F(S)$. Note that this also automatically proves the claim for tableaux on larger alphabets, since any JDT slide restricts to a JDT slide on the $(i,i+1)$-strip of a tableau.
																																																																															
																																																																															First, since $F(T)$ is dual equivalent to $T$, they slide to dual equivalent tableaux, so $\mathrm{slide}(F(T))$ is dual equivalent to $S$ and so also to $F(S)$.
																																																																															By Proposition \ref{prop:Haiman2}, it suffices to show that $\mathrm{slide}(F(T))$ and $F(S)$ also have the same rectification.
																																																																															
																																																																															Observe next that $F(S)$ and $\mathrm{slide}(F(T))$ have the same weight and rectification shape. This leaves at most two possibilities for their rectifications (see the diagrams in Figure \ref{fig:two-row}). In the case with two possibilities, the rectification shape has two rows and the distinguishing feature is whether or not $F'$ is defined, that is, whether or not the last $1$ or $2'$ in the word is a $2'$ (equivalently, appending a $1$ or $2'$ to the word should add a circled entry in the recording tableau). By the analysis in the proofs of Lemma \ref{lem:circling} and  Corollary \ref{cor:FE-dual-equiv}, this property is invariant under $F$. So, the following are equivalent:
																																																																															\begin{center}
																																																																																\begin{tabular}{rll}
																																																																																	$F'$ defined on $F(S)$
																																																																																	& $\Leftrightarrow F'$ defined on $S$ & (by the above)\\
																																																																																	& $\Leftrightarrow F'$ defined on $T$ & ($F'$ is coplactic)\\
																																																																																	& $\Leftrightarrow F'$ defined on $F(T)$ & (by the above)\\
																																																																																	& $\Leftrightarrow F'$ defined on $\mathrm{slide}(F(T))$ & ($F'$ is coplactic).
																																																																																\end{tabular}
																																																																															\end{center}
																																																																															Therefore, $F(S)$ and $\mathrm{slide}(F(T))$ have the same rectification, as desired.
																																																																														\end{proof}
																																																																														
																																																																														\noindent In the proof above, we showed that (in the case where $\rectify(T)$ has two rows) $F'$ is defined on $T$ if and only if it is defined on $F(T)$. In fact more is true:
																																																																														
																																																																														\begin{proposition} \label{prop:commutes}
																																																																															The operations $F, F', E, E'$ commute whenever both possible compositions are defined.
																																																																														\end{proposition}
																																																																														
																																																																														\begin{proof}
																																																																															By coplacticity, it suffices to prove this statement for words of straight shifted tableaux. The statement then follows from Propositions \ref{prop:rect-primed-action} and \ref{prop:rect-unprimed-action}. (See Figure \ref{fig:two-row} for an illustration.)
																																																																														\end{proof} 
																																																																														
																																																																														\begin{remark}
																																																																															In fact $F\circ F'$ is defined  on $w$ if and only if $F'\circ F$ is defined, and likewise with $E$ and $E'$. For $\{F,E'\}$ and $\{F',E\}$, the statement is instead that if both operators are defined on $w$, then so are both compositions.
																																																																														\end{remark}
																																																																														
																																																																														We also note that $F$ and $F'$ (and likewise $E$ and $E'$) coincide if and only if the rectification shape of the word has one row.
																																																																														
																																																																														\begin{corollary}
																																																																															If $\rectify(w)$ has only one row, then $F(w) = F'(w)$ and $E(w) = E'(w)$.  Otherwise, (when defined)
																																																																															$F(w)$ has the same number of primed symbols as $w$, and $F'(w)$ has one more primed symbol than $w$, which shows that they are different.
																																																																														\end{corollary}
																																																																														
																																																																														\section{Doubled crystal structure}\label{sec:crystal}
																																																																														
																																																																														We now consider the joint combinatorial structure built out of the primed and unprimed operators on words and tableaux from an arbitrary alphabet. Our main goal is to prove Theorem \ref{thm:main-doubled-typeA} on this `doubled' crystal structure.
																																																																														
																																																																														\subsection{Extending the alphabet: $E_i$, $E'_i$, $F_i$, and $F'_i$}
																																																																														
																																																																														We extend our operators to words and tableaux in the infinite alphabet $$\{1',1,2',2,3',3,\ldots\}.$$
																																																																														
																																																																														\begin{definition}
																																																																															%  For words in the alphabet $\{1',1,2',2,3',3,\ldots\}$, we
																																																																															We define $E_1$, $E'_1$, $F_1$, and $F'_1$ to be the result of applying the operators $E$, $E'$, $F$, and $F'$ respectively to the subword consisting of letters $1',1,2',2$.  We similarly define $E_i$, $E'_i$, $F_i$, and $F'_i$ to be the analogous operators on the restriction of the word to the letters $\{i,i+1,i',i+1'\}$.
																																																																														\end{definition}
																																																																														
																																																																														\begin{example}
																																																																															We have $F_1(2112'3)=212'23$ and $F_2(2112'3)=31123$.
																																																																														\end{example}
																																																																														
																																																																														We therefore have a sort of `doubled crystal' structure on tableaux of a given shape and rectification shape.  Two examples are shown in Figure \ref{fig:crystal}.
																																																																														
																																																																														\begin{figure}[t]
																																																																															\begin{center}
																																																																																\includegraphics[height=10cm]{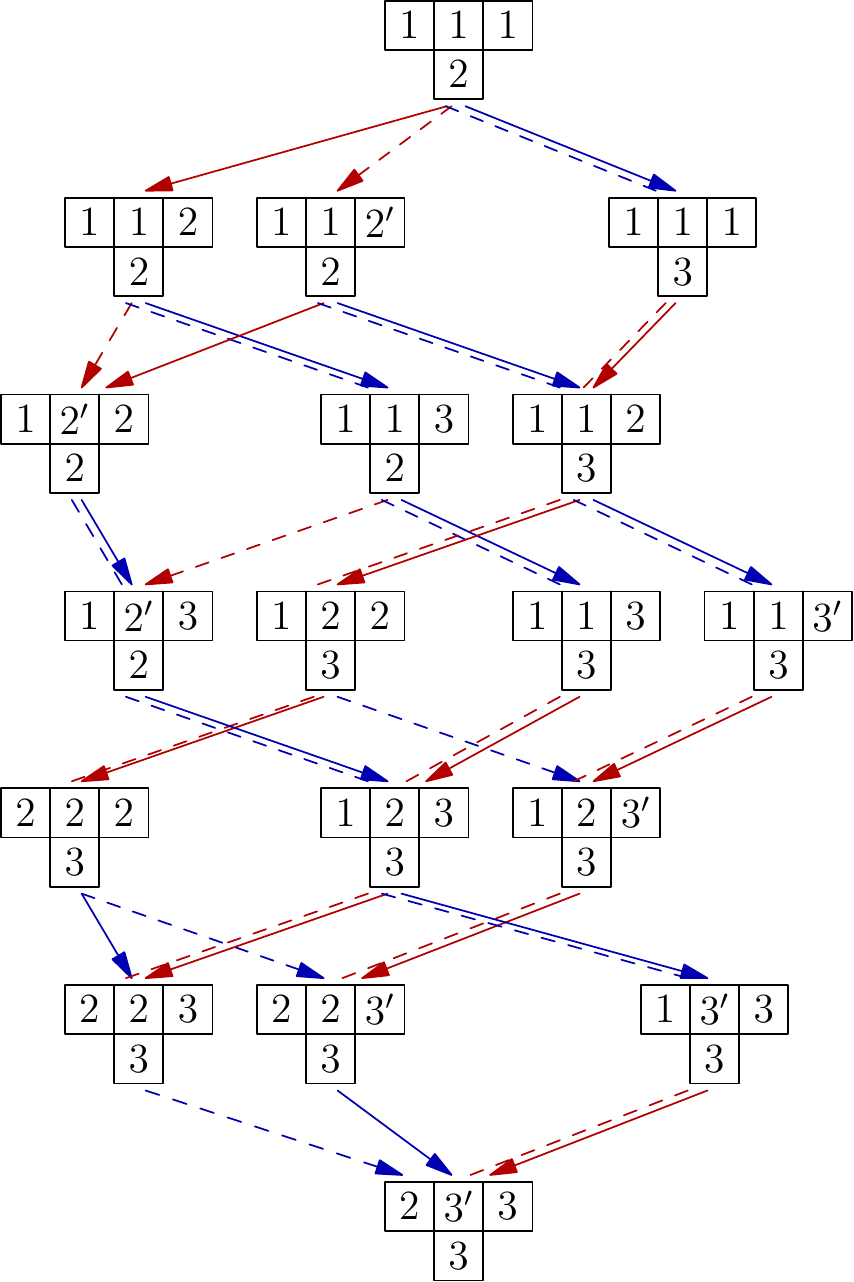} \hspace{1cm} \includegraphics[height=10cm]{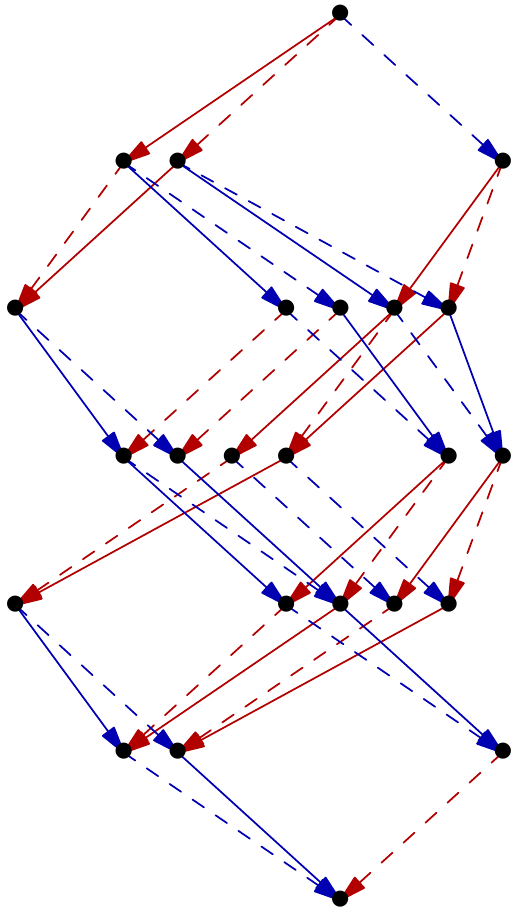}
																																																																															\end{center}
																																																																															\caption{\label{fig:crystal} {\bf Left:} The crystal graph for $\lambda = (3,1)$ with operators $F_i$, $F_i'$ for $i\in\{1,2\}$.  The $F_1$ and $F_1'$ arrows are respectively the red solid and dashed arrows (pointing leftwards), and $F_2$ and $F_2'$ are the blue solid and dashed arrows (pointing rightwards). {\bf Right:} The graph structure underlying the crystal for $\lambda = (4,2,1)$, with vertices arranged according to their weight (`clustered' vertices have the same weight).} 
																																																																														\end{figure}
																																																																														
																																																																														\begin{definition}
																																																																															Let $\ShST(\lambda/\mu,m)$ be the set of all shifted semistandard tableaux of shape $\lambda/\mu$ in the alphabet $\{1',1,\ldots,m',m\}$ (in canonical form).
																																																																														\end{definition}
																																																																														
																																																																														\begin{proposition} \label{prop:tableau-crystals}
																																																																															The unique element $T \in \ShST(\lambda,m)$ for which $E_i(T) = E'_i(T) = \varnothing$ for all $i$ is the tableau $T_{\mathrm{HIGH}}$, whose $i$-th row has all entries equal to $i$ (for all $i$).
																																																																															
																																																																															It follows that every element of $\ShST(\lambda,m)$ can be obtained from every other by some sequence of operations $E_i$, $E_i'$, $F_i$, and $F_i'$ (for various $i\in \{1,\ldots,m-1\}$).
																																																																														\end{proposition}
																																																																														
																																																																														We call the graph associated to $\ShST(\lambda,m)$ a {\bf shifted tableau crystal}. We will give a more precise definition in the next section.
																																																																														
																																																																														\begin{proof}
																																																																															Let $T \in \ShST(\lambda,m)$ be any tableau. It suffices to transform $T$ to $T_{\mathrm{HIGH}}$ using a sequence of $E_i,E_i'$ operators.
																																																																															
																																																																															If $T \ne T_{\mathrm{HIGH}}$, let $i$ be the smallest integer such that the $i$-th row contains an entry larger than $i$. Let $j$ be the largest entry of row $i$. The $(j-1,j)$ subword of $w$ is not ballot (since its last letter is $j$ or $j'$), hence by Proposition \ref{prop:ballot-iff-killed}, at least one of $E_{j-1}(T)$ or $E'_{j-1}(T)$ is nonzero and has higher weight. Since the set of possible weights is finite and $T_{\mathrm{HIGH}}$ is the unique tableau with weight $\lambda$, and this weight is highest, we must reach $T_{\mathrm{HIGH}}$ after a composition of finitely-many $E'$ and $E$ operators for various indices.
																																																																														\end{proof}
																																																																														
																																																																														The shifted tableaux crystals are `universal' in the following sense. Let $S$ be any finite set of words or tableaux in the alphabet $\{1',1, \ldots, m', m\}$, closed under $E_i, E'_i, F_i$ and $F'_i$ for $i = 1, \ldots, m-1$. (The main examples are $S = \ShST(\lambda/\mu,m)$, or the set of all words of a given length.) We will call $S$ a {\bf shifted word crystal}. Consider the graph structure on $S$ with edges labeled $i,i'$ corresponding to these operators.
																																																																														
																																																																														\begin{corollary} \label{cor:unique-highest-weight}
																																																																															Each connected component of $S$ has a unique highest-weight element $s^*$, and is isomorphic, as a weighted, edge-labeled graph, to the shifted tableau crystal $\ShST(\lambda,m)$, where $\lambda = \mathrm{wt}(s^*)$.
																																																																														\end{corollary}
																																																																														\begin{proof}
																																																																															We may assume $S$ is connected. By coplacticity, $S$ has the same weighted, labeled graph structure as the shifted tableau crystal obtained by replacing each $w \in S$ by its rectification. The statement now follows from Proposition \ref{prop:tableau-crystals}.
																																																																														\end{proof}
																																																																														
																																																																														A more precise statement is that the connected components of $S$ are shifted dual equivalence classes.
																																																																														
																																																																														\begin{corollary}
																																																																															Each connected component of $\B=\ShST(\lambda/\mu,m)$, or of a shifted word crystal, is a shifted dual equivalence class. 
																																																																														\end{corollary}
																																																																														
																																																																														\begin{proof}
																																																																															Since the crystal operators are coplactic, any two elements in the same connected component are dual equivalent. Then, by Corollary \ref{cor:unique-highest-weight}, the connected component contains the unique Littlewood-Richardson element of the dual equivalence class. Thus, conversely, if two elements are dual equivalent, they are connected (via the highest-weight element).
																																																																														\end{proof}
																																																																														
																																																																														\subsection{Embedded type A Kashiwara crystals}
																																																																														
																																																																														We now show that the primed and unprimed operators, considered independently, form type A Kashiwara crystals. We first recall the definition of an abstract Kashiwara crystal over the $\mathrm{GL}_m$ root system.  (Note: these are defined for general root systems in \cite{Kashiwara}, and we state the restricted definition below.)
																																																																														
																																																																														\begin{definition}
																																																																															A \textbf{Kashiwara crystal} (for $\mathrm{GL}_m$) is a nonempty set $\B$ together with maps $$e_i,f_i:\B\to \B\cup \{\varnothing\},$$
																																																																															$$\varepsilon_i,\varphi_i:\B\to \mathbb{Z}\cup \{-\infty\},$$
																																																																															$$\mathrm{wt}:\B\to \mathbb{Z}^m$$
																																																																															where $i\in \{1,\ldots,m-1\}$, satisfying the following axioms.
																																																																															\begin{enumerate}
																																																																																\item[K1.]  If $X,Y\in \B$ then $e_i(X)=Y$ if and only if $f_i(Y)=X$.  If this is the case then
																																																																																\begin{align*}
																																																																																	\varepsilon_i(Y)&=\varepsilon_i(X)-1 \\ \varphi_i(Y)&=\varphi_i(X)+1 \\
																																																																																	\mathrm{wt}(Y)&=\mathrm{wt}(X)+\alpha_i
																																																																																\end{align*}
																																																																																where $\alpha_i$ is the weight vector $(0,0,\ldots,0,1,-1,0,0,\ldots,0)$ having $1$ in the $i$th position and $-1$ in the $(i+1)$st position.  
																																																																																
																																																																																\item[K2.] For any $i\in \{1,\ldots,m-1\}$ and any $X\in \B$, we have $$\varphi_i(X)=\langle \mathrm{wt}(X),\alpha_i\rangle+\varepsilon_i(X)$$ where $\langle,\rangle$ is the standard dot product on vectors.
																																																																																
																																																																																\item[K3.] If $\varphi_i(X)=-\infty$ then $\varepsilon_i(X)=-\infty$ and vice versa.  Moreover if this is the case then $e_i(X)=f_i(X)=\varnothing$.
																																																																															\end{enumerate}
																																																																														\end{definition}
																																																																														
																																																																														\noindent We will use the same auxiliary functions for two overlapping crystal structures.
																																																																														
																																																																														\begin{definition}\label{def:crystal}
																																																																															Let $\B=\ShST(\lambda/\mu,m)$. For $T\in \B$, we let $\mathrm{wt}(T)$ be the weight, and for $i=1, \ldots, m-1$ we let
																																																																															\[(\varphi_i(T),\varepsilon_i(T)) := (x^{(i)},y^{(i)})\]
																																																																															be the coordinates of the endpoint of the $i,i+1$ lattice walk associated to $T$.\end{definition}

																																																																														\begin{proposition}
																																																																															The data of Definition \ref{def:crystal}, with $e_i=E_i$ and $f_i=F_i$ (or $e_i=E'_i$ and $f_i=F'_i$), satisfies the axioms of a Kashiwara crystal.
																																																																														\end{proposition}
																																																																														
																																																																														\begin{proof}
																																																																															Since $-\infty$ is not an output of either $\varepsilon_i$ or $\varphi_i$, we do not need to check axiom K3. For axiom K1, we know that $E_i$ and $F_i$ are partial inverses, as are $E'_i, F'_i$. The claim about the weight function follows from the definitions, and the claims about $\varphi_i,\varepsilon_i$ are just Corollary \ref{cor:F-on-walk} (for $E_i,F_i$) and Proposition \ref{prop:primed-lattice-endpoint} (for $E_i', F_i'$).
																																																																															
																																																																															Finally, axiom K2 follows directly from the definition of the lattice walk: we have
																																																																															\begin{align*}
																																																																																\langle \mathrm{wt}(T),\alpha_i\rangle &= \wt(T)_i-\wt(T)_{i+1} \\
																																																																																&= \#\{i',i \text{ steps}\} - \#\{i+1',i+1 \text{ steps}\},
																																																																															\end{align*}
																																																																															and we wish to show that this difference is equal to \[\varphi_i(T)-\varepsilon_i(T) = x^{(i)}-y^{(i)}.\] 
																																																																															But each $i$ or $i'$ step either increments $x$ or decrements $y$, and each $i+1$ or $i+1'$ does the opposite. Thus Axiom K2 is satisfied.
																																																																														\end{proof}
																																																																														
																																																																														\section{Characters and Schur $Q$-functions}\label{sec:characters}
																																																																														
																																																																														We now define characters of our crystals, recovering many of the combinatorial properties of Schur $Q$-functions (and their duals, Schur $P$-functions).
																																																																														
																																																																														\begin{definition}
																																																																															The \textbf{character} of a word $\hat{w}$, or of a tableau $T$ with reading word $\hat{w}$, is $$\sum_{w\in \hat{w}}x^{\mathrm{wt}(w)},$$ where $x^{\alpha}=x_1^{\alpha_1}x_2^{\alpha_2}\cdots $ for any tuple $\alpha$. This is just $2^{k}x^{\mathrm{wt}(\hat{w})}$, where $k$ is the number of nonzero parts of $\mathrm{wt}(w)$. The \textbf{character} of a collection of words or tableaux is the sum of the characters of its entries.
																																																																														\end{definition}
																																																																														
																																																																														The Schur $Q$-functions and Schur $P$-functions are certain specializations of Hall-Littlewood polynomials which live in $\Lambda_{\mathbb{Q}}(X)=\Lambda_{\mathbb{Q}}(x_1,x_2,\ldots)$, the ring of symmetric functions over $\mathbb{Q}$ (see \cite{Macdonald} for a thorough introduction to symmetric function theory).  More precisely, they are dual bases for the subring $\Gamma\subset \Lambda_{\QQ}(X)$ generated by the odd-degree power sum symmetric functions $$p_k(X)=x_1^k+x_2^k+x_3^k+\cdots,\hspace{1cm} k=1,3,5,7,\ldots.$$  (See \cite{Stembridge}.)  First defined by Schur \cite{Schur}, the Schur $Q$- and $P$-functions were later shown to exhibit the following combinatorial formulas in \cite{Stembridge}.
																																																																														
																																																																														\begin{definition}
																																																																															Let $\lambda/\mu$ be a shifted skew shape.  Define $\ShST_Q(\lambda/\mu,m)$ to be the set of all shifted semistandard tableaux of shape $\lambda/\mu$ in alphabet $\{1',1,\ldots,m',m\}$ in which the canonical form restriction is lifted, i.e., $i'$ is allowed at the start of the $i,i'$-subword.  Also define $\ShST_P(\lambda/\mu)$ to be the set of all shifted semistandard tableaux of shape $\lambda/\mu$ with entries from the alphabet $\{1',1,\ldots,m',m\}$ in which primes are not allowed on the staircase diagonal.
																																																																														\end{definition}
																																																																														
																																																																														\begin{definition}\label{def:SchurQ}
																																																																															The \textbf{Schur $Q$-function} $Q_{\lambda/\mu}$ is the symmetric function given by 
																																																																															\[Q_{\lambda/\mu}(X)=\sum_{T\in \ShST_Q(\lambda/\mu)} x^{\mathrm{wt}(T)}\]
																																																																															and the \textbf{Schur $P$-function} $P_{\lambda/\mu}$ is given by 
																																																																															\[P_{\lambda/\mu}(X)=\sum_{T\in \ShST_P(\lambda/\mu)} x^{\mathrm{wt}(T)}.\]
																																																																														\end{definition}
																																																																														
																																																																														\begin{proposition}
																																																																															The character of $\B = \ShST(\lambda/\mu)$ is the Schur $Q$-function $Q_{\lambda/\mu}(x)$.
																																																																														\end{proposition}
																																																																														
																																																																														\begin{proof}
																																																																															This follows immediately from the definition of the character, Definition \ref{def:SchurQ}, and Proposition \ref{prop:tableau-crystals}.
																																																																														\end{proof}
																																																																														
																																																																														It was shown in \cite{Stembridge} that the Schur $Q$-functions satisfy the following Littlewood-Richardson-type rule:
																																																																														$$Q_\mu Q_\nu=\sum 2^{\ell(\mu)+\ell(\nu)-\ell(\lambda)}f^{\lambda}_{\mu\nu}Q_\lambda$$
																																																																														where $f^{\lambda}_{\mu\nu}$ is the number of skew shifted Littlewood-Richardson tableaux (this definition requires the reading word to be in canonical form) of shape $\lambda/\mu$ and weight $\nu$.   It is easy to see that this is equivalent to the rule $$P_\mu P_\nu=\sum f^{\lambda}_{\mu\nu}P_\lambda$$ as well.  Since $Q$ and $P$ are dual under the Hall inner product on $\Lambda_\mathbb{Q}$, these coefficients also appear in the expansion of skew Schur $Q$-functions in terms of straight shapes: 
																																																																														\begin{equation} \label{eqn:skew-LR-schurQ}
																																																																														Q_{\lambda/\mu} = \sum_\nu f_{\mu\nu}^\lambda Q_\nu.
																																																																														\end{equation}
																																																																														
																																																																														By Corollary \ref{cor:unique-highest-weight} and \ref{prop:ballot-iff-killed}, our crystal graphs give a combinatorial interpretation of this last equation in terms of the connected components of the crystal for shape $\lambda/\mu$.
																																																																														
																																																																														\begin{corollary}
																																																																															By decomposing $\B=\ShST(\lambda/\mu,m)$ into its connected components, we get an isomorphism
																																																																															\[\ShST(\lambda/\mu,m) = \bigsqcup_\nu \ShST(\nu,m)^{f_{\mu,\nu}^\lambda}.\]
																																																																															Comparing characters recovers equation \eqref{eqn:skew-LR-schurQ}.
																																																																														\end{corollary}
																																																																														
																																																																														We also obtain a new combinatorial proof of symmetry of the Schur $Q$-functions.
																																																																														
																																																																														\begin{corollary} \label{cor:symmetry-schurQ}
																																																																															The function $Q_{\lambda/\mu}(x_1,x_2,\ldots)$ is symmetric in the variables $x_i$.
																																																																														\end{corollary}
																																																																														
																																																																														\begin{proof}
																																																																															It suffices to show that $Q_{\lambda/\mu}$ is symmetric under swapping $x_i,x_{i+1}$ for each $i$.  Consider the crystal $\B=\ShST(\lambda/\mu)$ whose character is $Q_{\lambda/\mu}$, and consider only the operators $F_i$ and $F_i'$ for a single fixed value of $i$.
																																																																															
																																																																															These operators decompose $\B$ into a disjoint union of two-row and one-row diagrams for $F_i,F_i'$.  Within a given such diagram, the weights $x^{\wt(T)}$ of the entries $T$ are constant except for the exponents of $x_i$ and $x_{i+1}$.  Factoring out the other variables, we see that in two-row diagrams, we obtain a symmetric sum of monomials $$4x_i^a x_{i+1}^b+8x_i^{a-1}x_{i+1}^{b+1}+\cdots+ 8x_i^{b+1}x_{i+1}^{a-1}+4x_i^{b}x_{i+1}^a$$ for some $a$ and $b$, and in one-row diagrams we obtain a symmetric sum of monomials $$2x_i^a x_{i+1}^b + 4x_i^{a-1}x_{i+1}^{b+1}+\cdots+ 4x_i^{b+1}x_{i+1}^{a-1}+2x_i^{b}x_{i+1}^a.$$  Thus the entire polynomial is symmetric in $x_i$ and $x_{i+1}$ as desired.
																																																																														\end{proof}
																																																																														
																																																																														\section{Uniqueness of shifted tableau crystals}
																																																																														\label{sec:uniqueness}
																																																																														
																																																																														We end by proving a strengthened form of Corollary \ref{cor:unique-highest-weight}: we show that the crystals $\B = \ShST(\lambda,n)$ are determined by their local combinatorial structure -- specifically the interactions between the $i,i',j,j'$ operators. In particular, we show that any connected graph with the same local interactions as these is isomorphic to some $\ShST(\lambda,n)$. The analogous statement for ordinary (non-shifted) tableaux is due to Stembridge \cite{Stembridge-typeA}, and the proofs in this section are modeled on his, though our local axioms are less explicit (see axiom (A4) below).
																																																																														
																																																																														We formalize the statement as follows. Let $G$ be a finite directed graph with vertices weighted by $\mathbb{Z}^n$ and edges labeled $i,i'$ for $1 \leq i \leq n-1$. As usual, let $\alpha_i$ be the vector
																																																																														\[\alpha_i = (0, \ldots, 0, 1,-1, 0, \ldots, 0)\]
																																																																														with $1$ in the $i$-th spot. Consider the following axioms that may apply to $G$:
																																																																														
																																																																														\begin{itemize}
																																																																															\item[(A1)] \textit{Edge labelings.} For each vertex $v \in G$ and index $i \in \{1, \ldots, n-1\}$ there is at most one incoming and one outgoing edge labeled $i$, and at most one labeled $i'$. If $v \to w$ is an $i$ or $i'$ edge, then $\wt(w) = \wt(v) - \alpha_i$, where $\wt$ is the weight.
																																																																														\end{itemize}
																																																																														
																																																																														Supposing (A1) holds, $G$ has the structure of a poset with covering relations indicated by the edges, and for each vertex $v$ we can define $f_i(v)$ and $f_i'(v)$ to be the unique next vertex along an $i$ or $i'$ edge respectively if it exists, and $\varnothing$ otherwise.  We similarly define $e_i, e_i'$ to be the (partial) inverse operations.  
																																																																														
																																																																														\begin{remark}
																																																																															In what follows, we will sometimes write $f_{i'}$ for $f_i'$ and $e_{i'}$ for $e_i'$.
																																																																														\end{remark}
																																																																														
																																																																														We also let $\hat\varepsilon_i(v),\hat\varphi_i(v)$ denote the distances to the ends of the $i$-string of $v$, i.e., $$\hat\varphi_i(v)=\max\{k: f_i^k(v)\neq \varnothing\}\hspace{0.5cm}\text{ and }\hspace{0.5cm}\hat\varepsilon_i(v)=\max\{k: e_i^k(v)\neq \varnothing\}.$$
																																																																														We define $\varphi'_i(v)$ and $\varepsilon_i'(v)$ analogously, using $f_i'$ and $e_i'$. 
																																																																														
																																																																														Next, we describe the interactions of the $i,i',j,j'$ edges when $|i-j| \ne 1$.
																																																																														\begin{itemize}
																																																																															\item[(A2)] \textit{Chains.} For each $v$ and each $i$, the vertices connected to $v$ by $i,i'$ edges collectively form either a two-row grid
																																																																															\[\xymatrix{
																																																																																\bullet \ar[r]^i \ar[d]_{i'} & \bullet \ar[r]^i \ar[d]_{i'}  & \cdots \ar[r]^i & \bullet \ar[d]_{i'} \\
																																																																																\bullet \ar[r]^i & \bullet \ar[r]^i & \cdots \ar[r]^i & \bullet
																																																																															}\]
																																																																															or a single row with coinciding $i$ and $i'$ edges:
																																																																															\[\xymatrix{
																																																																																\bullet \ar@<2pt>[r]^i \ar@<-2pt>[r]_{i'} & \bullet \ar@<2pt>[r]^i \ar@<-2pt>[r]_{i'} & \cdots \ar@<2pt>[r]^i \ar@<-2pt>[r]_{i'} & \bullet.
																																																																															}\]
																																																																														\end{itemize}
																																																																														When (A2) holds, we define $\varphi_i(v), \varepsilon_i(v)$ to be the total distance from $v$ to the highest- and lowest-weight vertices of its $(i,i')$-connected component. Thus for a two-row grid,
																																																																														$$\varphi_i(v) = \hat\varphi_i(v) + \varphi'_i(v), \hspace{1cm} \varepsilon_i(v) = \hat\varepsilon_i(v) + \varepsilon'_i(v)$$
																																																																														whereas for the single row
																																																																														$$\varphi_i(v) = \hat\varphi_i(v) = \varphi'_i(v), \hspace{1cm} \varepsilon_i(v) = \hat\varepsilon_i(v) = \varepsilon'_i(v).$$

																																																																														\begin{itemize}
																																																																															\item[(A3)] \textit{Nonadjacent indices commute.} If $|i-j| > 1$, then
																																																																															\[\raisebox{3.5ex}{\xymatrix{
																																																																																	\bullet \ar@{-}[r]^a \ar@{-}[d]_b & \bullet \\
																																																																																	\bullet
																																																																																}} \qquad \text{implies} \qquad \raisebox{3.5ex}{\xymatrix{
																																																																																\bullet \ar@{-}[r]^a \ar@{-}[d]_b & \bullet \ar@{-}[d]_b \\
																																																																																\bullet \ar@{-}[r]^a & \bullet
																																																																															}}\]
																																																																															for any $a \in \{i,i'\}$ and $b \in \{j,j'\}$. Here, the absence of arrowheads means that the $x$ and $y$ edges may be oriented in either direction.  In other words, if $f_a(v)\neq \varnothing$ and $f_b(v)\neq\varnothing$ then $f_bf_a(v)=f_af_b(v)\neq \varnothing$, and similarly for $e_a$ and $e_b$.
																																																																														\end{itemize}
																																																																														\begin{remark}
																																																																															Note that axioms (A2)-(A3) imply the first condition of Theorem \ref{thm:uniqueness-main}, namely, if $|i-j|>1$, then each connected component formed by the $i',i,j',j$-edges is a product of one- or two-row chains.
																																																																														\end{remark}
																																																																														
																																																																														Finally, we ask that the $i,i'$ and $i+1,i+1'$ edges have the same interactions as in the shifted tableau crystals $\ShST(\lambda,3)$ (on the alphabet $\{1',1,2',2,3',3\}$). To be precise, for each $i$, let $G^{(i,i+1)}$ be the subgraph obtained by deleting all but the $i',i,i+1',i+1$ edges, and subtracting $i-1$ from the edge labels. Let $\wt^{(i,i+1)}(v) \in \mathbb{Z}^3$ be the truncation of the weight vector to the $i,i+1,i+2$ spots.
																																																																														\begin{itemize}
																																																																															\item[(A4)] \textit{Local crystal structure}. For each $i$ and each connected component $C \subset G^{(i,i+1)}$, there is a strict partition $\lambda$ so that $C \cong \ShST(\lambda,3)$ as a weighted, edge-labeled graph (using $\wt^{(i,i+1)}$).
																																																																														\end{itemize}
																																																																														We will see that this implicit condition suffices to make $G$ globally isomorphic to a shifted tableau crystal. 
																																																																														A natural question is whether (A4) can be replaced by a more explicit set of axioms:
																																																																														
																																																																														\begin{question} \label{question:finitely-many-axioms}
																																																																															Is there a finite list of relations between the $i,i',i+1,i+1'$ edges, depending only on the weight and length functions, that implies axiom (A4)?
																																																																														\end{question}
																																																																														
																																																																														We now prove:
																																																																														
																																																																														\begin{theorem} \label{thm:unique-max}
																																																																															Let $G$ be a finite connected graph satisfying (A1)-(A4). Then $G$ has a unique maximal element $g^*$, and $\wt(g^*)$ is a shifted partition (possibly with trailing zeros).
																																																																														\end{theorem}
																																																																														
																																																																														\begin{proof}[Proof of Theorem \ref{thm:unique-max}]
																																																																															First observe that $G$, viewed as a poset, has the following property: whenever two distinct elements $y_1, y_2$ cover a common element $z$, say 
																																																																															$$f_a(y_1)=f_b(y_2)=z,$$ then $y_1, y_2$ have a common upper bound. (If $\{a,b\} = \{i,i'\}$ for some $i$, this follows from the `grid' case of axiom (A2); if $|a-b| > 1$, it follows from axiom (A3), and if $|a-b| = 1$, it follows from the fact that $y_1, y_2, z$ are in the same connected component of $G^{(i,i+1)}$, which  has a common maximum by axiom (A4) and Proposition \ref{prop:tableau-crystals}.)
																																																																															
																																																																															But any finite connected poset $P$ with this property has a unique global maximum. To see this, let $X = \{$ vertices comparable to two or more maxima $\} \subset P$. If $X \ne \varnothing$, let $x$ be a maximal element of $X$, and let $y,y'$ be covers of $x$ which are comparable to distinct (and unique since $y,y' \notin X$) maximal elements of $P$. But, $y,y'$ must also have a common upper bound, a contradiction. Thus $X = \varnothing$; by connectedness, we're done.
																																																																															
																																																																															Finally, we check that $\lambda = \wt(g^*)$ is a strict partition (possibly with trailing zeros). Observe that $g^*$ is maximal when viewed as part of each subgraph $G^{(i,i+1)}$. So, axiom (A4) and Proposition \ref{prop:tableau-crystals} implies that $\lambda_i > \lambda_{i+1}$ (unless both are $0$).
																																																																														\end{proof}
																																																																														
																																																																														\begin{theorem} \label{thm:unique-iso}
																																																																															Any two finite connected graphs $G,H$ satisfying (A1)-(A4), and with the same highest-weight vector $\lambda = \wt(g^*) = \wt(h^*)$, are canonically isomorphic.
																																																																														\end{theorem}
																																																																														
																																																																														Since we have shown that $\ShST(\lambda,n)$ satisfies (A1)-(A4), we conclude:
																																																																														
																																																																														\begin{corollary} \label{cor:iso-to-crys}
																																																																															If $G$ is a finite connected graph satisfying (A1)-(A4), with highest-weight element $g^*$, there is a canonical isomorphism $G \cong \ShST(\lambda,n)$, where $\lambda = \wt(g^*)$.
																																																																														\end{corollary}
																																																																														
																																																																														\begin{proof}[Proof of Theorem \ref{thm:unique-iso}]
																																																																															We build the isomorphism $T : G \to H$ inductively. Explicitly, for each $r \geq 0$, we set
																																																																															\[G_{\leq r} = \{g \in G \text{ reachable in at most } r \text{ steps from } g^*\},\]
																																																																															and we define $H_{\leq r}$ similarly. We build compatible isomorphisms
																																																																															\[T : G_{\leq r} \xrightarrow{\ \sim\ } H_{\leq r}\]
																																																																															that preserve the weight and the functions $\varepsilon_i, \varepsilon'_i, \hat\varepsilon_i$ and $\varphi_i, \varphi'_i, \hat\varphi_i$ for all $i$. Since $G$ and $H$ are finite, we are done by taking $r$ large enough.
																																																																															
																																																																															For $r=0$, the isomorphism identifies the highest-weight elements, and axiom (A4) determines $\varphi_i,\varphi'_i,\hat\varphi_i, \varepsilon_i,\varepsilon'_i,\hat\varepsilon_i$ for each $i$. %Explicitly, we have
																																																																															%\[(\hat\varphi_i,\varphi_i',\varphi_i) = \begin{cases}
																																																																															%(\lambda_i-\lambda_{i+1}-1,1,\lambda_i-\lambda_{i+1}) & i < n \\
																																																																															%(\lambda_i,\lambda_i,\lambda_i) & i \geq n
																																																																															%\end{cases}\]
																																																																															%(Explicitly, for $i < n$ we have $\varphi'_i = 1, \hat\varphi_i = \lambda_i - \lambda_{i+1} - 1$, and $\varphi_i = \lambda_i - \lambda_{i+1}$; for $i=n$ we have $\varphi'_i = \hat\varphi_i = \varphi_i = \lambda_n$, and for $i > n$ we have $\varphi'_i = \hat\varphi_i = \varphi_i = 0$.)
																																																																															%(Of course, $\hat\varepsilon_i = \varepsilon'_i = \varepsilon_i = 0$ for the highest weight element.)
																																																																															
																																																																															Now suppose $T$ exists for $r-1$. We must extend $T$ to an isomorphism $G_{\leq r}\rightarrow H_{\leq r}$. We note that, since $T$ preserves all the length functions, $G_{\leq r-1}$ and $H_{\leq r-1}$ have the same set of outward-pointing edges for each edge label.  We can map the end nodes of these edges to the corresponding nodes in $H$ in a well-defined way as long as, whenever two outwards edges point to the same element of $G_{\leq r} \setminus G_{\leq r-1}$, their isomorphic images in $H_{\leq r} \setminus H_{\leq r-1}$ also point to a single element. 
																																																																															
																																																																															To show that this holds, suppose we have
																																																																															\[\xymatrix{
																																																																																g_1 \ar[dr]_a && g_2 \ar[dl]^b &&
																																																																																T(g_1) \ar[dr]_a && T(g_2) \ar[dl]^b \\
																																																																																& g_3 &&&& ?
																																																																															}\]
																																																																															Note that if $g_1 = g_2$ and $\{a,b\} = \{i,i'\}$ for some $i$ (a `double edge'), then the isomorphic image in $H$ is again a double edge, since by axiom (A2) those are detected by the length functions (at the top of the $i,i'$-string).  In addition, axiom (A1) forces any double edge to be of this form.  So we may assume $g_1 \ne g_2$. 
																																																																															
																																																																															Let $i,j$ be the numerical values of $a,b$.
																																																																															
																																																																															{\bf Case 1}: $i = j$ or $|i - j| > 1$. Then by (A2)-(A3) there is a covering element $t \in G_{\leq r-2}$, forming a square. Hence, on the $H$ side, $T(t)$ covers $T(g_1), T(g_2)$, and a second application of (A2)-(A3) forces the square to commute in $H$.
																																																																															
																																																																															{\bf Case 2}: $|i-j| = 1$. Without loss of generality $j = i+1$. On the $G$ side, we see that $g_1, g_2, g_3$ are in the same connected component $C$ of $G^{(i,i+1)}$. Therefore they are all dominated by the local high-weight element $\tilde{g} \in C$, and $\tilde{g} \in G_{\leq r-2}$ (by Theorem \ref{thm:unique-max}, and since $G$ is graded). Therefore $T(g_1)$ and $T(g_2)$ are in the same local component in $H$, with local maximum $T(\tilde{g})$. The local components are canonically isomorphic because $\wt(\tilde{g}) = \wt(T(\tilde{g}))$ 
																																																																															(and have no nontrivial automorphisms), hence the edges must connect.
																																																																															
																																																																															Having constructed $T : G_{\leq r} \to H_{\leq r}$, we only need to check that it preserves the weight and length functions. Let $g \in G_{\leq r} \setminus G_{\leq r-1}$. By hypothesis, there is an edge $\tilde{g} \to g$, say of numerical value $i$. By axiom (A1) and induction,
																																																																															\[\wt(g) = \wt(\tilde{g}) - \alpha_i = \wt(T(\tilde{g})) - \alpha_i = \wt(T(g)).\]
																																																																															Next, by axioms (A2)-(A3), the $j$-lengths of $g$ equal those of $\tilde{g}$ for all $j \neq i\pm 1$. By induction, the same holds for $T(g)$ and  $T(\tilde{g})$. Finally, for $j = i\pm 1$, we again apply axiom (A4). The $i\pm 1$ lengths at $g$ and $T(g)$ must equal those of the corresponding element of the tableau crystal isomorphic to the local connected component.
																																																																														\end{proof}

																																																																													\end{document}